\documentclass[11pt]{article}

\usepackage{listings}
\usepackage{pdflscape}
\usepackage{epsfig}
\usepackage{lscape}
\usepackage{longtable}
\usepackage{amsmath,amssymb,amsthm}
\usepackage{graphicx}
\usepackage{color}
\usepackage{placeins}
\usepackage{url}
\usepackage{cases}
\usepackage{longtable}
\usepackage{color}
\def\norm#1{\|#1 \|}
\def\inprod#1#2{\langle#1,\,#2 \rangle}
\def\cA{{\cal A}}
\def\cC{{\cal C}}
\def\cX{{\cal X}}
\def\cY{{\cal Y}}
\def\cL{{\cal L}}
\def\cI{{\cal I}}
\def\cO{{\cal O}}
\def\cT{{\cal T}}

\def\cM{{\cal M}}
\def\cU{{\cal U}}

\def\cJ{{\cal J}}
\def\cW{{\cal W}}
\def\dist{{\textup{dist}}}
\def\nn{{\nonumber}}
\def\idch{{\textup{int}(\textup{dom}\, h^*)}}
\def\dcp{{\textup{dom}\, p^*}}
\def\SSNAL {{\sc Ssnal }}
\def\SSNCG {{\sc Ssn }}
\newtheorem{theorem}{Theorem}
\newtheorem{assumption}{Assumption}
\newtheorem{remark}{Remark}
\newtheorem{definition}{Definition}
\newtheorem{proposition}{Proposition}

\newcommand{\red}[1]{\begin{color}{red}#1\end{color}}

\def\mc{\multicolumn}
\usepackage{url}            % simple URL typesetting
\usepackage{tikz}
\usetikzlibrary{arrows, decorations.markings}
\tikzstyle{vecArrow} = [thick, decoration={markings,mark=at position
	1 with {\arrow[semithick]{open triangle 60}}},
double distance=2pt, shorten >= 5.5pt,
preaction = {decorate},
postaction = {draw,line width=1.4pt, white,shorten >= 4.5pt}]
\tikzstyle{innerWhite} = [semithick, white,line width=1.4pt, shorten >= 4.5pt]

\title{A highly efficient semismooth Newton   augmented Lagrangian method for solving Lasso problems\footnote{An earlier version of this paper was  made available in arXiv as arXiv:1607.05428. 	This research is
		supported in part by the Ministry of Education,
		Singapore, Academic Research Fund under Grant R-146-000-194-112.}}

\author{Xudong Li\thanks{Department of Mathematics, National University of Singapore, 10 Lower Kent Ridge Road, Singapore ({\tt matlixu@nus.edu.sg}).
},
\; Defeng Sun\thanks{Department  of  Mathematics  and  Risk  Management  Institute, National University of Singapore, 10 Lower Kent Ridge Road, Singapore ({\tt matsundf@nus.edu.sg}).
}
\  and \ Kim-Chuan Toh\thanks{Department of Mathematics, National University of Singapore, 10 Lower Kent Ridge Road, Singapore
({\tt mattohkc@nus.edu.sg}).
}
}
\date{April 27, 2017}

\begin{document}

\maketitle

\begin{abstract}
We develop a fast and robust algorithm for solving large scale convex composite optimization models with an emphasis on the $\ell_1$-regularized least squares regression (Lasso) problems. Despite the fact that there exist a large number of solvers in the literature for the Lasso problems, we found that no solver can efficiently handle difficult large scale regression problems with real data. By leveraging  on available error bound results
to realize the  asymptotic superlinear convergence property of the augmented Lagrangian algorithm, and by exploiting the second order sparsity  of the problem through the semismooth Newton method, we are able to propose an  algorithm, called {\sc Ssnal}, to efficiently solve the aforementioned difficult problems. Under very mild conditions, which hold automatically for Lasso problems, both the primal and the dual iteration sequences generated by {\sc Ssnal} possess a fast linear convergence rate, which can even be superlinear asymptotically. Numerical comparisons between our approach and a number of state-of-the-art solvers, on real data sets, are presented to demonstrate the high efficiency and robustness of our proposed algorithm in solving difficult large scale Lasso problems.

%\vspace{0.2cm}
%\noinden

\end{abstract}


\section{Introduction}
In this paper, we aim to  
{design}
 a highly efficient and robust algorithm for solving convex composite optimization problems including  the following $\ell_1$-regularized least squares (LS)
problem
\begin{eqnarray}
  \quad \min \left\{\frac{1}{2}\norm{\cA x -b}^2 + \lambda\norm{x}_1 \right\},
\label{eq-l1}
\end{eqnarray}
where $\cA:\cX\to\cY$ is a linear map whose adjoint is denoted as $\cA^*$, $b\in\cY$ and $\lambda >0$ are given data, and  $\cX$, $\cY$ are two real finite dimensional Euclidean spaces each equipped with an inner product $\inprod{\cdot}{\cdot}$ and its induced norm $\norm{\cdot}$.

With the {advent of convenient} automated data collection technologies, the Big Data era brings new challenges in analyzing massive data due to the inherent sizes -- large samples and high dimensionality \cite{Fan2014challenge}. In order to {respond to} these challenges, researchers have developed many
{new statistical tools to analyze such data}. Among these,
the $\ell_1$-regularized models are {arguably} the most intensively studied.
They are used in many applications, such as in compressive sensing, high-dimensional variable selection and image reconstruction, etc. Most notably, the model (\ref{eq-l1}), named as Lasso, was proposed in \cite{Tibshirani1996regression} and has been  used heavily in high-dimensional statistics and machine learning.
The model (\ref{eq-l1}) has also been studied in
{ the signal processing context under the name of basis pursuit denoising} \cite{chen1998atomic}.
In addition to its own importance in statistics and machine learning, Lasso problem \eqref{eq-l1} also appears as an inner subproblem of many important algorithms.
For example, {in recent papers \cite{Berg2008probing, Aravkin2016}}, a  level set method was proposed to solve a computationally more challenging reformulation of the Lasso problem, i.e.,
\[\min \left\{\norm{x}_1\,\mid\, \frac{1}{2}\norm{\cA x - b}^2 \le \sigma\right\}.\]
%This
{The level set}
method
relies critically on the assumption that the following optimization problem, the same type as \eqref{eq-l1},
\[\min \left\{ \frac{1}{2}\norm{\cA x - b}^2 + \delta_{B_\lambda}(x)\right\}, \quad
B_\lambda = \{ x\mid \norm{\cdot}_1\le \lambda\},
\]
can be efficiently solved to high accuracy. The Lasso-type  optimization problems also appear as subproblems  in   various proximal Newton methods for solving convex composite optimization problems \cite{Byrd2016an,Lee2014proximal,Yue2016inexact}. Notably, in a broader sense, all these proximal Newton methods belong to  the class of algorithms studied in \cite{Fischer2002local}.

The above mentioned importance together with a wide range of applications of (\ref{eq-l1}) 
{has inspired many researchers to develop various algorithms}
for solving this problem and its equivalent reformulations.
 These algorithms can  roughly be divided into two categories. The first category consists of   algorithms that use only  the gradient  information, for example, accelerated proximal gradient (APG) method \cite{Nesterov1983a,Beck2009a},  GPSR \cite{Figueiredo2007gradient}, SPGL1 \cite{Berg2008probing}, SpaRSA \cite{Wright2009sparse},   FPC\_AS \cite{wen2010fast}, and NESTA \cite{Becker2011nesta}, to name only a few. Meanwhile, algorithms in the second category, including but not limited to mfIPM \cite{Fountoulakis14matrix}, SNF \cite{Milzarek14}, BAS \cite{Byrd2016a}, SQA \cite{Byrd2016an}, OBA \cite{Keskar2016a}, FBS-Newton \cite{Xiao2016semismooth},   utilize the second order information of the underlying problem in the algorithmic design to accelerate the convergence.
 {Nearly all of these second order information based solvers rely on certain nondegeneracy assumptions to guarantee the non-singularity of the corresponding inner linear systems.
 The only exception is
 the inexact interior-point-algorithm based solver mfIPM, which does not rely on the {nondegeneracy} assumption but require the availability of  appropriate pre-conditioners
to ameliorate the extreme ill-conditioning in the linear systems of the subproblems.}
 For nondegenerate problems, the solvers in the second category  generally work   quite well and usually outperform the  algorithms in the first category when high {accuracy} solutions are sought. 
 In this paper, we also aim to solve  the Lasso problems by making use of the second order  information. The novelty of our approach is that we do not need any {nondegeneracy}  assumption in our theory or computations. The core idea is to  analyze the fundamental nonsmooth structures in the problems  to formulate and  solve specific semismooth equations  with well conditioned symmetric and positive definite generalized Jacobian matrices, which consequently play a critical role {in}
 our algorithmic design.     When applied to solve  difficult large scale  sparse optimization problems,
{even for degenerate ones},
our approach can outperform the first order algorithms by a huge margin 
{regardless of
 whether  low or high accuracy solutions are sought}.
This is in a sharp   contrast to most of the other second order based solvers mentioned above,
{where their competitive advantages over first-order methods only become apparent when
high accuracy solutions are sought.}

Our proposed algorithm is a semismooth Newton augmented Lagrangian method {(in short, {\sc Ssnal})} for solving the dual of problem (\ref{eq-l1}) where the 
{sparsity property}
of the second order generalized Hessian is wisely 
{exploited}. This algorithmic framework is adapted from those appeared
in \cite{SDPNAL,jiang2014partial,YangST2015,QSDPNAL} for solving semidefinite programming problems where impressive numerical results have been reported. Specialized to the vector case, our {\sc Ssnal} possesses unique features that  are  not available in the semidefinite programming case. {It is these
combined unique features that allow} our algorithm
to converge at a very fast speed with very low computational costs at each step.  
{Indeed,}
for large scale sparse Lasso problems,
  our numerical experiments show that    {the proposed}
algorithm  needs at most {a few} dozens of outer iterations to  reach  solutions with the desired accuracy while all the inner  
{semismooth Newton subproblems}
can be solved very cheaply.
One reason for this  
{impressive}
performance is  that   the piecewise linear-quadratic structure of the Lasso problem \eqref{eq-l1} guarantees the
 asymptotic superlinear convergence of the augmented Lagrangian algorithm. Beyond the piecewise linear-quadratic case, we also study more general functions to guarantee this fast convergence rate.
More importantly, since there are several   desirable  properties including the strong convexity of the objective function  in the inner subproblems,  in each outer iteration we only need to execute a few (usually one to four)  semismooth Newton steps to solve the underlying subproblem. As will be shown later, for Lasso problems with sparse optimal solutions, the computational costs of performing these semismooth Newton steps can be made to be extremely cheap compared to other costs.
This seems to  be counter-intuitive  as normally one would  expect a  second order method
{to be computationally much more expensive than}
the first order methods at each step.
 Here, we make {this counter-intuitive  achievement} possible by carefully 
{exploiting}
 the second order sparsity   in the augmented Lagrangian  functions. Notably, our algorithmic framework not only works for models such as  Lasso, adaptive Lasso \cite{alasso} and elastic net \cite{enet}, but can also be applied to more general convex composite optimization problems.
The high performance of our algorithm also
serves to show that the  second order information, more specifically the nonsmooth second order information, can be and should be  incorporated  intelligently into the algorithmic design for large scale optimization problems.

The remaining parts of this paper are organized as follows. In the next section, we introduce some definitions and present preliminary results on the metric subregularity {of multivalued mappings}. We should emphasize here that these stability results play a pivotal role in the analysis of the  convergence rate of our algorithm. In Section 3, we propose an augmented Lagrangian algorithm to solve the general convex composite optimization model and analyze its asymptotic superlinear convergence. The semismooth Newton algorithm for solving the inner subproblems and the efficient implementation of the algorithm are also presented in this section. In Section 4,  we conduct extensive numerical experiments to
evaluate the performance of {\sc Ssnal} in solving various
Lasso problems. We conclude our paper in the final section.

%%%%%%%%%%%%%%%%%%%%%%%%%%%%%%%%%%%%%
\section{Preliminaries}
We discuss in this section some stability properties of convex composite optimization problems.
{It will become apparent later that
these} stability properties are {the} key ingredients for {establishing} the fast convergence of our augmented Lagrangian method.

Recall that $\cX$ and $\cY$ are 
two real finite dimensional Euclidean spaces.
For a given closed proper convex function $p:\cX\to(-\infty,+\infty]$, the proximal mapping $\textup{Prox}_p(\cdot)$ associated with $p$   is defined by
\[
\textup{Prox}_p(x) := \arg\min_{u\in \cX} \Big\{p(x) + \frac{1}{2}\|u-x\|^2\Big\}, \quad   \forall x\in \cX.
\]
We will often make use of the following Moreau identity
$
\textup{Prox}_{t p}( x) + t \textup{Prox}_{p^*/t}(x/t) = x,
$
where $t > 0$ is a given parameter. Denote $\textup{dist}(x,C) = \inf_{x'\in C}\norm{x-x'}$ for any $x\in \cX$ and any set $C\subset \cX$.

Let $F:\cX \rightrightarrows \cY$ be a multivalued mapping.
We define the graph of $F$ to be the set
\[{\rm gph} F: = \{(x,y)\in\cX\times\cY\mid y\in F(x)\}.\]
$F^{-1}$, the inverse of $F$, is the multivalued mapping from $\cY$ to $\cX$ whose graph is
$\{(y,x)\mid(x,y)\in {\rm gph} F\}$.
\begin{definition}[Error bound]\label{def:err-bound}
  Let $F:\cX \rightrightarrows \cY$ be a multivalued mapping and $y\in\cY$ satisfy $F^{-1}(y)\ne \emptyset$. $F$ is said to satisfy the error bound condition  for the point $ y$ with  modulus $\kappa \ge 0$, if  there exists $\varepsilon >0$  such that
  if $x\in\cX$ with ${\rm dist}( y,F(x)) \le \varepsilon$ then
  \begin{eqnarray}\label{def:eq-err-bound}
  {\rm dist}(x,F^{-1}( y )) \le \kappa \, {\rm dist}(  y, F(x)).
  \end{eqnarray}
\end{definition}
The above error bound condition was called the growth condition in \cite{luque1984asymptotic} and was used to analyze the local linear convergence properties of the proximal point algorithm.
Recall that $F:\cX\rightrightarrows\cY$ is called a polyhedral multifunction if its graph is
%a
{the}
union of finitely many polyhedral convex sets. In \cite{robinson1981some}, Robinson established the following celebrated proposition on the error bound result for polyhedral multifunctions.
\begin{proposition}\label{thm:error_bound_polyhedral}
	Let $F$ be a polyhedral multifunction from $\cX$ to $\cY$.
Then, $F$ satisfies the error bound condition \eqref{def:eq-err-bound} for any point $y\in\cY$ satisfying $F^{-1}(y)\ne\emptyset$ with a  common modulus $\kappa \ge 0$.
\end{proposition}

For  later uses, we present the following definition of %the
{metric}  subregularity from Chapter 3 in  \cite{Dontchev2009}.
\begin{definition} [Metric subregularity]
	\label{def:subregular}
Let $F:\cX\rightrightarrows\cY$ be a multivalued mapping and $(\bar x,\bar y)\in  {\rm gph} F$.
	$F$  is said to be metrically subregular at $\bar x$ for $\bar y$ with modulus $\kappa \ge 0$, if there exist neighborhoods $U$ of $\bar x$ and $V$ of $\bar y$ such that
	\begin{equation*}
	\label{metric_sub}
	{\rm dist}(x, F^{-1}(\bar y))\le \kappa\, {\rm dist}(\bar y, F(x)\cap V), \quad \forall x\in U.
	\end{equation*}
\end{definition}
	From the above {definition}, we see that if $F:\cX\rightrightarrows\cY$ satisfies the error bound condition \eqref{def:eq-err-bound} for $\bar y$ with the modulus $\kappa$,  then it is metrically subregular at $\bar x$ for $\bar y$ with the same modulus $\kappa$ for any $\bar x\in F^{-1}(\bar y)$.

The following {definition on} essential smoothness is taken from \cite[Section 26]{rockafellar1970convex}.
\begin{definition}[Essential smoothness]
  A proper convex function $f$ {on $\cX$} is essentially smooth if $f$ is differentiable on $\textup{int}\,(\textup{dom}\,f) \neq \emptyset$
and $\lim_{{k\to\infty}}\norm{\nabla f(x^k)} = +\infty$ whenever $\{x^k\}$ is a sequence in $\textup{int}\,(\textup{dom}\,f)$ converging to a boundary point $x$ of $\textup{int}\,(\textup{dom}\,f)$.
\end{definition}
{In particular, a smooth convex function on $\cX$ is essentially smooth. Moreover, if a closed proper convex function $f$ is strictly convex on ${\rm dom}\, f$, then its conjugate $f^*$ is essentially smooth \cite[Theorem 26.3]{rockafellar1970convex}.}

Consider the following composite convex optimization model
\begin{equation}
  \label{prob:p_composite_op}
   \max - \left\{f(x): = h(\cA x) - \inprod{c}{x}  + p(x) \right\},
\end{equation}
where $\cA : \cX\to\cY$ is a linear map, $h:\cY\to\Re$ and $p:\cX\to(-\infty, +\infty]$ are two closed proper convex functions, $c\in \cX$ is a given vector. The dual of \eqref{prob:p_composite_op} can be written as
\begin{equation}
  \label{prob:d_composite_op}
  \min \left\{ h^*(y) + p^*(z) \, \mid \, \cA^*y + z = c\right\},
\end{equation}
where $g^*$ and $p^*$ are the Fenchel conjugate functions of
$g$ and $h$, respectively.
Throughout this section, we make the following blanket assumption on $h^*(\cdot)$ and $p^*(\cdot)$.

\begin{assumption}
\label{ass:chcp}
$h^*(\cdot)$ is essentially smooth and $p^*(\cdot)$ is either an indicator function $\delta_{P}(\cdot)$ or a support function $\delta^*_{P}(\cdot)$ for some nonempty polyhedral convex set $P\subseteq \cX$.
%Moreover,
{Note that}
$\nabla h^*$ is locally Lipschitz continuous and  directionally differentiable  on $\textup{int}\,(\textup{dom}\,h^*)$.
\end{assumption}
Under Assumption \ref{ass:chcp}, by \cite[Theorem 26.1]{rockafellar1970convex}, we know that $\partial h^*(y) =\emptyset$ whenever $y\not\in \textup{int}\,(\textup{dom}\,h^*)$.
Denote by $l$ the Lagrangian function for \eqref{prob:d_composite_op}:
\begin{equation}
  \label{Lag-D}
  l(y,z,x) = h^*(y) + p^*(z) - \inprod{x}{\cA^*y + z-c}, \quad \forall\,(y,z,x)
  \in\cY\times\cX\times\cX.
\end{equation}
Corresponding to the closed proper convex function $f$ in the objective of \eqref{prob:p_composite_op} and the convex-concave function $l$ in {\eqref{Lag-D}},  define the maximal monotone operators $\cT_f$ and $\mathcal{T}_l$ \cite{rockafellar1976augmented}, by
\[\cT_f(x) := \partial f(x),\quad \mathcal{T}_l(y,z,x) := \{(y',z',x') \,|\, (y',z',-x')\in\partial l(y,z, x)\},\]
whose   inverse are given, respectively, by
\[
\cT_f^{-1}(x') := \partial f^*(x'), \quad
\mathcal{T}^{-1}_l(y',z',x') := \{(y,z,x)\,|\,(y',z',-x')\in\partial l(y,z, x)\}.
\]
{Unlike the case for $\cT_f$ \cite{Luo1992on,Tseng2009a,Zhou2017unified}, stability results of $\cT_l$ which correspond to the perturbations of both primal and dual solutions are very limited.}
Next, as a tool for studying the convergence rate of {\sc Ssnal}, we shall establish a theorem which reveals the metric subregularity of $\cT_l$ under some mild assumptions.

The KKT system associated with problem \eqref{prob:d_composite_op} is given as follows:
\begin{equation*}
  \label{KKT_dual_org}
  0 \in \partial h^*(y) - \cA x, \quad
 0 \in -x + \partial p^*(z), \quad
  0 = \cA^*y  + z -c,
  \quad (x,y,z)\in\cX\times\cY\times\cX.
\end{equation*}
Assume that the above KKT system has at least one solution.
This existence assumption together with the essentially smoothness assumption on $h^*$ implies that the above KKT system can be equivalently rewritten as
\begin{equation}
  \label{KKT_dual}
  0 = \nabla h^*(y) - \cA x, \quad
 0 \in -x + \partial p^*(z), \quad
  0 = \cA^*y  + z -c,
  \quad (x,y,z)\in\cX\times\idch\times\cX.
\end{equation}
Therefore, under the above assumptions, we only need to focus on $\textup{int}\,(\textup{dom}\,h^*)\times\cX$ when solving problem \eqref{prob:d_composite_op}.
 Let $(\bar y,\bar z)$ be an optimal solution to problem \eqref{prob:d_composite_op}. Then, we know that the set of the Lagrangian multipliers associated with $(\bar y, \bar z)$, denoted as $\cM(\bar y,\bar z)$, is nonempty.
Define the critical cone associated with \eqref{prob:d_composite_op} at $(\bar y, \bar z)$ as follows:
\begin{equation}
  \label{eq:critical_general_d}
  \cC(\bar y,\bar z): = \left\{(d_1,d_2) \in \cY\times\cX \mid
  \cA^*d_1 + d_2 = 0, \inprod{\nabla h^*(\bar y)}{d_1} + (p^*)'(\bar z; d_2) = 0, d_2 \in \cT_{{\rm dom}(p^*)}(\bar z)\right\},
\end{equation}
where $(p^*)'(\bar z; d_2)$ is the directional derivative of $p^*$ at $\bar z$ with respect to $d_2$, $\cT_{{\rm dom}(p^*)}(\bar z)$ is the tangent cone of ${\rm dom}(p^*)$ at $\bar z$. When the conjugate function $p^*$ is taken to be the indicator function of a nonempty polyhedral set $P$,  the above definition reduces directly to the standard definition of the critical cone in the nonlinear programming setting.
\begin{definition} [Second order sufficient condition]
Let $(\bar y, \bar z) \in \cY\times\cX$ be an optimal solution to problem \eqref{prob:d_composite_op} with $\cM(\bar y,\bar z) \ne \emptyset$. We say that the second order sufficient condition
for problem \eqref{prob:d_composite_op} holds at $(\bar y,\bar z)$ if
\[ \inprod{d_1}{(\nabla h^*)'(\bar y; d_1)} >0, \quad \forall\, 0 \ne (d_1,d_2)\in\cC(\bar y,\bar z).\]
\end{definition}
By building on the  proof ideas from the literature on nonlinear programming problems \cite{Dontchev1998,Klatte2000upper,Izmailov2012a},   we are able to prove the following result on the  metric subregularity of $\cT_l$. This allows us to prove the linear and even the  asymptotic superlinear convergence of the sequences generated by the {\sc Ssnal} {algorithm to be presented in the next section even} when the objective in problem \eqref{prob:p_composite_op} does not possess the piecewise linear-quadratic structure as in the Lasso problem \eqref{eq-l1}.

\begin{theorem}\label{thm:metric_sub_tl}
 Assume that the KKT system \eqref{KKT_dual} has at least one solution and denote it as $(\bar x,\bar y, \bar z)$. Suppose that Assumption  {\rm\ref{ass:chcp}} holds and that the second order sufficient condition for problem \eqref{prob:d_composite_op} holds at $(\bar y, \bar z)$.
 Then, $\cT_{l}$ is metrically subregular at $(\bar y,\bar z,\bar x)$ for the origin.
\end{theorem}
\begin{proof}
	First, we  claim  that there exists a neighborhood $\cU$ of $(\bar x, \bar y, \bar z)$ such that for any $w = (u_1, u_2, v) \in \cW: = \cY\times\cX\times\cX$ with $\norm{w}$ sufficiently small, any solution $(x(w),y(w),z(w))\in \cU$ of the perturbed KKT system
	 \begin{equation}\label{eq:perturbed_KKT}
	 0 = \nabla h^*(y) - \cA x - u_1, \quad 0 \in -x - u_2 + \partial p^*(z), \quad 0 = \cA^*y + z - c - v
	 \end{equation}
	 satisfies the following estimate
	 \begin{equation}
	 \label{eq:Lipyz}
	 \norm{(y(w),z(w)) - (\bar y,\bar z)} =  O(\norm{w}).
	 \end{equation}
For the sake of contradiction, suppose that our claim is not true.
Then, there exist some sequences $\{w^k { := (u_1^k,u_2^k,v^k)} \}$ and $\{(x^k,y^k,z^k)\}$ such that $w^k\to 0$, $ (x^k,y^k,z^k)  \to (\bar x,\bar y,\bar z)$, for each $k$ the point $(x^k,y^k,z^k)$ is a solution of \eqref{eq:perturbed_KKT} for $w = w^k$, and
\[  \norm{(y^k,z^k) - (\bar y,\bar z)} > \gamma_k \norm{w^k}\]
with some $\gamma_k >0$ such that $\gamma_k\to \infty$.
Passing onto a subsequence if necessary, we  assume that $\{\big((y^k,z^k) - (\bar y, \bar z)\big) / \norm{(y^k,z^k) - (\bar y, \bar z)}\}$ converges to some $\xi = (\xi_1,\xi_2)\in \cY\times\cX$, $\norm{\xi} = 1$. Then, setting $t_k = \norm{(y^k,z^k) - (\bar y, \bar z)} $ and passing to a   subsequence further if necessary, by the local Lipschitz continuity and the directional differentiability of $\nabla h^* (\cdot)$ at $\bar y$,
we know that for all $k$ sufficiently large
\begin{align*}
\nabla h^*(y^k) - \nabla h^*(\bar y) = &{} \nabla h^*(\bar y + t_k \xi_1) - \nabla h^*(\bar y) + \nabla h^*(y^k) - \nabla h^*(\bar y + t_k \xi_1) \\[5pt]
 = &{} t_k (\nabla h^*)'(\bar y; \xi_1) + o(t_k) + O(\norm{y^k - \bar y - t_k \xi_1})\\[5pt]
 = &{} t_k (\nabla h^*)'(\bar y; \xi_1) + o(t_k).
\end{align*}
Denote for each $k$, $\hat x^k := x^k + u_2^k$.
Simple calculations show that for all $k$ sufficiently large
\begin{equation} \label{eq:yxi1}
 0 = \nabla h^*(y^k) - \nabla h^*(\bar y) - \cA ( \hat x^k - \bar x) + \cA u_2^k - u_1^k
 = t_k (\nabla h^*)'(\bar y; \xi_1) + o(t_k) - \cA( \hat x^k - \bar x)
\end{equation}
and
\begin{equation} \label{eq:Atxi_1pxi_2}
0 = \cA (y^k - \bar y) + (z^k - \bar z) - v^k = t_k (\cA^*\xi_1 + \xi_2) + o(t_k). 
\end{equation}
Dividing both sides of equation \eqref{eq:Atxi_1pxi_2} by $t_k$ and then taking limits, we obtain
\begin{equation}\label{eq:atxi1xi2}
\cA^*\xi_1 + \xi_2 = 0,
\end{equation}
which further implies that
\begin{equation}
  \label{eq:gradchd1}
  \inprod{\nabla h^*(\bar y)}{\xi_1} =
  \inprod{\cA\bar x}{\xi_1} = -\inprod{\bar x}{\xi_2}.
\end{equation}
 Since
$\hat x^k \in \partial p^*(z^k)$, we know that for all $k$ sufficiently large,
\[z^k = \bar z + t_k \xi_2 + o(t_k) \in \dcp. \] That is, $\xi_2 \in \cT_{\dcp}(\bar z)$.

According to the structure of $p^*$, we separate our discussions into two cases.

{\bf Case I:}
There exists a nonempty polyhedral convex set $P$ such that $p^*(z) = \delta^*_P(z),\,\forall z\in\cX$. Then, for each $k$, it holds that
\[\hat x^k = \Pi_{P}(z^k + \hat x^k).\]
By \cite[Theorem 4.1.1]{facchinei2003finite}, we have
\begin{equation}\label{eq:ppc}
\hat x^k = \Pi_{P}(z^k + \hat x^k) = \Pi_{P}(\bar z + \bar x)
+ \Pi_{\cC}(z^k - \bar z + \hat x^k - \bar x)
= \bar x + \Pi_{\cC}(z^k - \bar z + \hat x^k - \bar x),
\end{equation}
where $\cC$ is the critical cone of $P$ at $\bar z + \bar x$, i.e.,
\[\cC \equiv \cC_P(\bar z + \bar x):= \cT_P(\bar x) \cap \bar z ^{\perp}.\]
Since $\cC$ is a polyhedral cone, we know from \cite[Proposition 4.1.4]{facchinei2003finite} that
$\Pi_{\cC}(\cdot)$ is a piecewise linear function, i.e., there exist  a positive integer $l$ and orthogonal projectors $B_1,\ldots,B_l$ such that for any $x\in\cX$,
\[\Pi_{\cC}(x) \in \left\{ B_1 x,\ldots, B_l x\right\}.\]
By restricting to {a}
 subsequence if necessary, we may further assume that there exists $1\le j' \le l$ such that for all $k\ge 1$,
\begin{equation}\label{eq:pl}
\Pi_{\cC}(z^k - \bar z + \hat x^k - \bar x) = B_{j'} (z^k - \bar z + \hat x^k - \bar x) = \Pi_{{\rm Range}(B_{j'})}(z^k - \bar z + \hat x^k - \bar x).
\end{equation}
Denote $L:={\rm Range}(B_{j'})$.
Combining \eqref{eq:ppc} and \eqref{eq:pl}, we get
\begin{equation*}\label{eq:critialzx}
\cC\cap L  \owns  (\hat x^k - \bar x) \perp (z^k - \bar z)   \in \cC^\circ\cap L^{\perp},
\end{equation*}
where $\cC^\circ$ is the polar cone of $\cC$.
{Since $\hat{x}^k-\bar{x} \in \cC$, we have}
$\inprod{\hat x^k - \bar x}{\bar z} = 0$, which, together with
$\inprod{\hat x^k - \bar x}{z^k - \bar z} = 0$, implies
$\inprod{\hat x^k - \bar x}{z^k} = 0.$
%Since,
{Thus}
for all $k$ sufficiently large,
\[\inprod{z^k}{\hat x^k} = \inprod{z^k}{\bar x} = \inprod{\bar z + t_k \xi_2 + o(t_k)}{\bar x},\]
%we have that
{and it follows that}
\begin{equation}
\label{eq:cpdizxi2}
\begin{aligned}
(p^*)'(\bar z; \xi_2) ={}& \lim_{k \to \infty} \frac{\delta_P^*(\bar z + t_k\xi_2) - \delta^*_P(\bar z)}{t_k}
=  \lim_{k \to \infty} \frac{\delta_P^*(z^k) - \delta^*_P(\bar z)}{t_k} \\[5pt]
={}& \lim_{k \to \infty}\frac{\inprod{z^k}{\hat x^k} - \inprod{\bar z}{\bar x}}{t_k}
= \inprod{\bar x}{\xi_2}.
\end{aligned}
\end{equation}
By \eqref{eq:atxi1xi2}, \eqref{eq:gradchd1} and \eqref{eq:cpdizxi2},  we know that
$(\xi_1,\xi_2) \in \cC(\bar y, \bar z)$. Since $\cC\cap\cL$ is a polyhedral convex cone, we know from \cite[Theorem 19.3]{rockafellar1970convex} that $\cA(\cC\cap\cL)$ is also a polyhedral convex cone, which, together with \eqref{eq:yxi1}, implies
\[(\nabla h^*)'(\bar y; \xi_1) \in \cA(\cC\cap L).\] Therefore, there exists a vector $\eta \in \cC\cap L$ such that
\[\inprod{\xi_1}{(\nabla h^*)'(\bar y; \xi_1)} = \inprod{\xi_1}{\cA \eta} = -\inprod{\xi_2}{\eta} = 0,\]
where the last equality %is due
{follows from the fact that $\xi_2\in \cC^\circ\cap L^{\perp}$.
Note that the last inclusion holds since the polyhedral convex cone
$\cC^\circ\cap L^{\perp}$ is closed } and \[\xi_2 = \lim_{k\to\infty} \frac{z^k - \bar z}{t_k} \in \cC^\circ \cap L^{\perp}.\]
{As $0\neq \xi = (\xi_1,\xi_2) \in \cC(\bar{y},\bar{z})$,
but $\inprod{\xi_1}{(\nabla h^*)'(\bar y; \xi_1)} =0$, this contradicts the assumption that the
second order sufficient condition holds for \eqref{prob:d_composite_op} at $(\bar{y},\bar{z})$.
Thus we have proved our claim
%\eqref{eq:Lipyz}
for Case I.}

{\bf Case II:} There exists a nonempty polyhedral convex set $P$ such that $p^*(z) = \delta_P(z),\,\forall z\in\cX$. Then, we know that for each $k$,
\[z^k = \Pi_{P}(z^k + \hat x^k) .\]
Since $(\delta_P)'(\bar z;d_2) = 0, \forall d_2\in\cT_{{\rm dom}(p^*)}(\bar z)$, the critical cone in \eqref{eq:critical_general_d} now takes the following form
\[  \cC(\bar y,\bar z) = \left\{(d_1,d_2) \in \cY\times\cX \mid
\cA^*d_1 + d_2 = 0, \inprod{\nabla h^*(\bar y)}{d_1} = 0, d_2 \in \cT_{{\rm dom}(p^*)}(\bar z)\right\}.\]
Similar to Case I, without loss of generality, we can assume that there exists a subspace $L$ such that for all $k\ge 1$,
\begin{equation*}%\label{eq:critialzx}
\cC\cap L  \owns  (z^k - \bar z) \perp (\hat x^k - \bar x) \in \cC^\circ\cap L^{\perp},
\end{equation*}
where \[\cC \equiv \cC_P(\bar z + \bar x):= \cT_P(\bar z) \cap \bar x ^{\perp}.\] Since $\cC\cap L$ is a polyhedral convex cone, we know
\[\xi_2 = \lim_{k\to \infty} \frac{z^k - \bar z}{t_k} \in \cC\cap L,\]
and consequently $\inprod{\xi_2}{\bar x} = 0$,
which, together with \eqref{eq:atxi1xi2} and \eqref{eq:gradchd1}, implies $\xi = (\xi_1,\xi_2) \in \cC(\bar y,\bar z)$. By \eqref{eq:yxi1} and the fact that $\cA(\cC^\circ\cap\cL^{\perp})$ is a polyhedral convex cone, we know that
\[(\nabla h^*)'(\bar y; \xi_1) \in \cA(\cC^\circ\cap L^{\perp}).\]
Therefore, there exists a vector $\eta \in \cC^\circ\cap L^{\perp}$ such that
\[{\inprod{\xi_1}{(\nabla h^*)^\prime (\bar y; \xi_1)}} = \inprod{\xi_1}{\cA \eta} = -\inprod{\xi_2}{\eta} = 0.\]
Since $\xi = (\xi_1,\xi_2) \neq 0$, we arrive at a contradiction {to the assumed second
order sufficient condition}. So our claim is also true for this case.

In summary, we have proven that there exists a neighborhood $\cU$ of $(\bar x, \bar y, \bar z)$ such that for any $w$ close enough to the origin, equation \eqref{eq:Lipyz} holds for any solution $(x(w),y(w),z(w))\in \cU$ to the perturbed KKT system \eqref{eq:perturbed_KKT}. Next we show that $\cT_l$ is metrically subregular at $(\bar y,\bar z, \bar x)$ for the origin.

Define the mapping $\Theta_{KKT}: \cX\times\cY\times\cX\times\cW \to \cY\times\cX\times\cX$ by
\begin{equation*}
\Theta_{KKT}(x,y,z,w): =  \left(
\begin{array}{c}
\nabla h^*(y) - \cA x - u_1 \\
z - {\rm Prox}_{p^*}(z + x + u_2) \\
\cA^*y + z -c- v \\
\end{array}
\right),\quad\forall (x,y,z,w)\in\cX\times\cY\times\cX\times\cW
\end{equation*}
and define the mapping $\theta:\cX\to\cY\times\cX\times\cX$ as follows:
\[\theta(x): = \Theta_{\rm KKT}(x,\bar y,\bar z, 0), \quad \forall x\in\cX.\]
Then, we have
$x\in\cM(\bar y,\bar z)$ if and only if $\theta(x)=0$. Since ${\rm Prox}_{p^*}(\cdot)$ is a piecewise affine function, $\theta(\cdot)$ is a piecewise affine function and thus a polyhedral multifunction. By using Proposition \ref{thm:error_bound_polyhedral}
and shrinking the neighborhood $\cU$ if necessary,  for any $w$ close enough to the origin and any solution $(x(w),y(w),z(w)) \in \cU$ of the perturbed KKT system \eqref{eq:perturbed_KKT}, we have
\begin{align*}
{\rm dist}(x(w), \cM(\bar y, \bar z)) ={}& O(\norm{\theta(x(w))}) \\[5pt]
={}& O(\norm{\Theta_{\rm KKT}(x(w),\bar y, \bar z, 0) - \Theta_{\rm KKT}(x(w), y(w), z(w), w)})\\[5pt]
= {}&O(\norm{w} + \norm{(y(w),z(w))-(\bar y, \bar z)}),
\end{align*}
which, together with \eqref{eq:Lipyz}, {implies the existence} of a constant $\kappa \ge 0$ such that
\begin{equation}\label{eq:upLip}
	 \norm{(y(w),z(w)) - (\bar y, \bar z)} + {\rm dist}(x(w),\cM(\bar y, \bar z)) \le \kappa \norm{w}. %\quad\forall\, (x(w),y(w),z(w)) \in \cU
\end{equation}
Thus, by  Definition \ref{def:subregular}, we have proven  that $\cT_l$ is metrically subregular at $(\bar y,\bar z,\bar x)$ for the origin. The proof of {the} theorem is completed.
\end{proof}
\begin{remark}\label{rem:tftl}
	For convex piecewise linear-quadratic programming problems such as the $\ell_1$ and elastic net regularized LS problem, we know from {\rm \cite{sun1986thesis}} and {\rm \cite[Proposition 12.30]{rockafellar1998variational}} that the corresponding operators $\mathcal{T}_l$ and $\cT_f$ are polyhedral multifunctions and thus, by  Proposition \ref{thm:error_bound_polyhedral}, the error bound condition holds.
	Here we emphasize again that the error bound condition holds for the Lasso problem {\rm(\ref{eq-l1})}.
	Moreover, $\cT_f$, associated with the $\ell_1$ or elastic net regularized logistic regression model, i.e., for a given vector $b\in\Re^m$, the loss function $h:\Re^m\to\Re$ in problem \eqref{prob:p_composite_op} takes the %following
 form
\[ h(y) = \sum_{i=1}^m \log(1 + e^{-b_i y_i}), \quad \forall y\in\Re^m,\]
also satisfies the error bound condition {\rm \cite{Luo1992on,Tseng2009a}}.
Meanwhile, from Theorem  {\rm\ref{thm:metric_sub_tl}}, we know that $\cT_l$ corresponding to the $\ell_1$ or the elastic net regularized logistic regression model is metrically subregular at any solutions to the KKT system
{\eqref{KKT_dual}}
 for the origin.
\end{remark}

%%%%%%%%%%%%%%%%%%%%%%%%%%%%%%%%%%%%%
\section{An augmented Lagrangian method with asymptotic superlinear convergence}
\label{FPPA}
Recall the general convex composite model \eqref{prob:p_composite_op}
\begin{equation}
({\bf P})\quad \max -\left\{f(x)= h(\cA x) - \inprod{c}{x} + p(x) \right\}
\nonumber %\label{eq-primal}
\end{equation}
and its dual
\[
({\bf D})  \quad  \min \{  h^*(y) + p^*(z) \,|\,
\cA^*y + z = c\}.
\]
In this section, we shall propose an asymptotically superlinearly convergent augmented Lagrangian method for solving ({\bf P}) and ({\bf D}).
In this section, we make the following standing assumptions regarding the function $h$.
\begin{assumption}
  \label{assumption:h}
  \begin{enumerate}
    \item [{\rm (a)}]   $h:\cY\to\Re$ is a convex differentiable function
{whose gradient is $1/\alpha_h$-Lipschitz continuous, i.e.,}
\[\norm{\nabla h(y') - \nabla h(y)}\le (1/\alpha_h) \norm{y' - y },\quad \forall y',y\in\cY.\]
\item [{\rm (b)}] $h$ is essentially locally strongly
convex {\rm \cite{Goebel2008local}}, i.e., for any compact and convex set $K\subset \textup{dom}\,\partial h$,
there exists
$\beta_{K} >0$ such that
\[(1-\lambda)h(y') + \lambda h(y) \ge h((1-\lambda)y' + \lambda y) + \frac{1}{2}\beta_K\lambda(1-\lambda)\norm{y' - y}^2,\quad \forall y',y\in K,\]
for all $\lambda \in[0,1]$.
  \end{enumerate}
\end{assumption}
Many {commonly used} loss functions 
 in the machine learning literature satisfy the above mild assumptions.
For example, $h$ can be the loss function in the linear regression,  logistic regression and  Poisson regression models. {While the strict convexity of a convex function is closely related to the differentiability of its conjugate, the essential local strong convexity in Assumption \ref{assumption:h}(b) 
for a convex function 
was first proposed in \cite{Goebel2008local} to obtain a characterization of the local Lipschitz continuity of the gradient of its conjugate function.}

 The aforementioned assumptions on $h$ imply the following useful properties of $h^*$. Firstly, by \cite[Proposition 12.60]{rockafellar1998variational}, we know that $h^*$ is strongly convex with modulus $\alpha_h$. Secondly, by \cite[Corollary 4.4]{Goebel2008local}, we know that $h^*$ is essentially smooth and $\nabla h^*$ is locally Lipschitz continuous on $\textup{int}\,(\textup{dom}\,h^*)$. If the solution set to the KKT system associated with ({\bf P}) and ({\bf D}) is further assumed to be nonempty, similar to what we have discussed in the last section,  one only needs to focus on $\textup{int}\,(\textup{dom}\,h^*)\times\cX$ when solving ({\bf D}).
Given $\sigma>0$, the augmented Lagrangian function associated with ({\bf D}) is given by
\[
\cL_{\sigma}(y,z;x) := l(y,z,x)
+\frac{\sigma}{2}\norm{\cA^*y + z - c}^2, \quad \forall\,(y,z,x)
  \in\cY\times\cX\times\cX,
  \]
where the Lagrangian function $l(\cdot)$ is defined in \eqref{Lag-D}.

\subsection{{\sc Ssnal}: A semismooth Newton augmented Lagrangian algorithm for ({\bf D})}
The detailed steps of our algorithm \SSNAL for solving ({\bf D}) are given as follows. Since a semismooth Newton method will be employed to solve the subproblems involved in this
prototype of the inexact augmented Lagrangian method \cite{rockafellar1976augmented},  
{it is natural for us to call
our algorithm as} {\sc Ssnal}.

\medskip

\centerline{\fbox{\parbox{\textwidth}{
{\bf Algorithm {\sc Ssnal}}: {\bf An inexact augmented Lagrangian method for ({\bf D}).}
\\[5pt]
Let  $\sigma_0 >0$ be a given parameter.
Choose $(y^0,z^0,x^0)\in\idch\times\textup{dom}\, p^*\times\cX$.
For $k=0,1,\ldots$, perform the following steps in each iteration:
\begin{description}
  \item [Step 1.] Compute
  \begin{equation} \label{p2:alm-sub}
   (y^{k+1},z^{k+1}) \approx \arg\min \{\Psi_k (y,z):= \cL_{\sigma_k}(y,z;x^k)\}.
  \end{equation}
     \item[Step 2.] Compute
     $x^{k+1} = x^k - \sigma_k(\cA^*y^{k+1} + z^{k+1} - c)$
  and update $\sigma_{k+1} \uparrow \sigma_\infty\leq \infty$ .
\end{description}
}}}

\medskip
Next, we shall adapt the
results developed in \cite{rockafellar1976monotone,rockafellar1976augmented} and
\cite{luque1984asymptotic} to establish the global and local superlinear convergence of our algorithm.

Since the inner problems
can not be solved exactly, we use the following standard stopping criterion studied in \cite{rockafellar1976monotone,rockafellar1976augmented} for approximately solving (\ref{p2:alm-sub}) 
$$
({\rm A})\quad \Psi_k(y^{k+1},z^{k+1}) - \inf \Psi_k \le \varepsilon_k^2/2\sigma_k,\quad \sum_{k=0}^{\infty} \varepsilon_k < \infty.
$$
Now, we can state the global convergence of Algorithm {\sc Ssnal} from \cite{rockafellar1976monotone,rockafellar1976augmented} without much difficulty.
\begin{theorem}\label{thm:golbal-converge-ppa}
Suppose that  Assumption {\rm\ref{assumption:h}} holds and that the solution set to $({\bf P})$ is nonempty. Let  $\{ (y^k, z^k, x^k)\}$  be the infinite sequence generated by Algorithm \SSNAL with stopping criterion $({\rm A})$. Then, the sequence $\{x^k\}$ is bounded and converges to an optimal solution of $({\bf P})$. In addition, $\{(y^k,z^k)\}$ is also bounded and converges to the unique optimal solution
$(y^*,z^*)\in\idch\times\dcp$ of $({\bf D})$.
\end{theorem}

\begin{proof}
Since the solution set to ({\bf P}) is assumed to be nonempty, the optimal value of ({\bf P}) is finite. From Assumption \ref{assumption:h}(a), we have that ${\rm dom}\,h = \cY$ and $h^*$ is strongly convex \cite[Proposition 12.60]{rockafellar1998variational}. Then, by Fenchel's Duality Theorem \cite[Corollary 31.2.1]{rockafellar1970convex}, we know that the solution set to ({\bf D}) is nonempty and the optimal value of ({\bf D}) is finite and equals to the optimal value of ({\bf P}). That is, the solution set to the KKT system associated with ({\bf P}) and ({\bf D}) is nonempty. The uniqueness of the optimal solution $(y^*,z^*)\in\idch\times\cX$ of ({\bf D}) then follows directly from the strong convexity of $h^*$. By combining this uniqueness with \cite[Theorem 4]{rockafellar1976augmented}, one can easily obtain the boundedness of $\{(y^k,z^k)\}$ and other desired results {readily.} 
\end{proof}

We need the the following stopping criteria for the local convergence analysis:
\begin{align*}
&({\rm B1})\quad \Psi_k(y^{k+1},z^{k+1}) - \inf \Psi_k \le (\delta_k^2/2\sigma_k) \norm{ x^{k+1} - x^k}^2, \quad
\sum_{k=0}^\infty\delta_k < +\infty, \\[5pt]
&({\rm B2})\quad\textup{dist}(0,\partial \Psi_k(y^{k+1},z^{k+1})) \le (\delta'_k/\sigma_k)\norm{ x^{k+1} - x^k}, \quad
0\le \delta'_k \to 0.
\end{align*}
where
$ x^{k+1} := x^k + \sigma_k (\cA^*y^{k+1} + z^{k+1}-c).$

\begin{theorem}\label{thm:local-linear}
Assume that Assumption {\rm\ref{assumption:h}} holds and that the solution set $\Omega$ to $({\bf P})$ is nonempty. Suppose that  $\cT_f$ satisfies the error bound
condition \eqref{def:eq-err-bound}  for the origin with modulus $a_f$.
Let  $\{ (y^k, z^k, x^k)\}$  be any infinite sequence generated by Algorithm \SSNAL with stopping criteria $({\rm A})$ and $({\rm B1})$. Then, the sequence $\{x^k\}$ converges to $x^*\in\Omega$ and for all $k$ sufficiently large,
\begin{equation}\label{eq:converge_x}
\dist(x^{k+1}, \Omega) \; \leq\;
\theta_k\dist(x^{k}, \Omega), 
\end{equation}
where { $\theta_k =\big(a_f(a_f^2 + \sigma_k^2)^{-1/2} + 2\delta_k \big)(1 - \delta_k)^{-1} \to \theta_\infty = a_f(a_f^2 + \sigma_{\infty}^2)^{-1/2} <1$} as $k\to+\infty$. Moreover, the sequence $\{(y^k,z^k)\}$ converges to the optimal unique solution $(y^*,z^*)\in\idch\times\dcp$ to $({\bf D})$.

Moreover, if $\cT_l$ is metrically subregular at $( y^*,z^*, x^*)$ for the origin with modulus $a_l$ and the stopping criterion $({\rm B2})$ is also used, then for all $k$ sufficiently large,
\begin{equation*}
\norm{(y^{k+1},z^{k+1}) - (y^*,z^*)} \le \theta_k' \norm{x^{k+1} - x^k}, 
\end{equation*}
where $\theta_k' = a_l(1+\delta_k')/\sigma_k$ with $\lim_{k\to \infty} \theta_k' = a_l/\sigma_{\infty}$.
%$\lim_{k\to \infty} \theta_k'(1 - \delta_k)^{-1} = a_l/\sigma_{\infty}$.
\end{theorem}
 \begin{proof}
 The first part of the theorem follows from  \cite[Theorem 2.1]{luque1984asymptotic},
 \cite[Proposition 7, Theorem 5]{rockafellar1976augmented} and Theorem \ref{thm:golbal-converge-ppa}.
To prove the second part, we recall that if $\cT_l$ is metrically subregular at $(y^*,z^*, x^*)$ for the origin with the modulus $a_l$ and $(y^k,z^k,x^k)\to (y^*,z^*,x^*)$,
then for all $k$ sufficiently large,
\[\norm{(y^{k+1},z^{k+1}) -(y^*,z^*)} + {\rm dist}(x^{k+1},\Omega) \le a_l\, {\dist }(0,\cT_l(y^{k+1},z^{k+1},x^{k+1})).\]
Therefore, by the estimate (4.21) in \cite{rockafellar1976augmented} and the  stopping criterion $({\rm B2})$, we obtain that for all $k$ sufficiently large,
\[\norm{(y^{k+1},z^{k+1}) -(y^*,z^*)} \le a_l(1 + \delta_k')/\sigma_k\norm{x^{k+1} -x^k}.\]
This completes the proof for Theorem \ref{thm:local-linear}.
\end{proof}

\begin{remark}
	\label{rmk:convergence_yz}
	{Recently advances in \cite{Cui2016} reveal the asymptotic
	R-superliner convergence of the dual iteration sequence $\{(y^k,z^k)\}$. 
	Indeed, from \cite[Proposition 4.1]{Cui2016}, under the same conditions of Theorem \ref{thm:local-linear}, we have that for $k$ sufficiently large, 
	\begin{equation}\label{eq:converge_yz}
	 \norm{(y^{k+1},z^{k+1}) - (y^*,z^*)} \le \theta_k' \norm{x^{k+1} - x^k} \le \theta_k' (1 -\delta_k)^{-1} {\rm dist}(x^k,\
	\Omega),
	\end{equation}
	where $\theta_k' (1 -\delta_k)^{-1} \to a_l/\sigma_{\infty}$.
	Then, if $\sigma_{\infty} = \infty$, inequalities \eqref{eq:converge_x}
	and \eqref{eq:converge_yz} imply that $\{x^k\}$ and $\{(y^k,z^k)\}$ converge Q-superlinearly  and  R-superlinearly, respectively.}
\end{remark}
{We should emphasize here that by
combining} Remarks \ref{rem:tftl} and \ref{rmk:convergence_yz} and Theorem \ref{thm:local-linear}, %we arrive at the conclusion that
{our
Algorithm {\sc Ssnal} is guaranteed to produce an asymptotically superlinearly convergent sequence when used to solve $({\bf D})$ for} many commonly used regularizers and loss functions.
{In particular, the {\sc Ssnal} algorithm is asymptotically superlinearly convergent when
applied to the dual of \eqref{eq-l1}.
}

%%%%%%%%%%%%%%%%%%%%%%%%%%%%%%%%%%
\subsection{Solving the augmented Lagrangian subproblems}

Here we shall propose an efficient semismooth Newton  algorithm to solve the inner subproblems %of the above
{in the}
augmented Lagrangian method (\ref{p2:alm-sub}). That is, for some fixed $\sigma >0$ and $\tilde x\in \cX$, we consider to solve
\begin{equation}
  \label{sub_alm}
  \min_{y,z}\; \Psi(y,z): = \cL_{\sigma}(y,z; \tilde x).
\end{equation}
Since $\Psi(\cdot,\cdot)$ is a strongly convex function, we have that, for any
$\alpha \in \Re$, the level set $\cL_{\alpha}:=\{(y,z)
\in \textup{dom}\, h^* \times\textup{dom}\, p^* \,\mid\,
\Psi(y,z)\le \alpha\}$ is a closed and bounded convex set.
Moreover, problem (\ref{sub_alm}) admits a unique optimal solution denoted as $(\bar y,\bar z)\in \textup{int}(\textup{dom}\, h^*)\times \textup{dom}\, p^*$.

Denote, for any $y\in\cY$,
\begin{align*}
 \psi(y) :={}& \inf_{z} \Psi (y,z)  \\
={}& h^*(y)
+ p^*(\textup{Prox}_{p^*/ \sigma}(\tilde x/\sigma - \cA^*y + c))
+ \frac{1}{2\sigma}\norm{\textup{Prox}_{\sigma p}(\tilde x - \sigma (\cA^*y - c))}^2 - \frac{1}{2\sigma}\norm{\tilde x}^2.
\end{align*}
Therefore, if
$
  (\bar y,\bar z) = \arg\min \Psi(y,z),
$
then $(\bar y, \bar z) \in\idch\times\dcp$ can be computed simultaneously by
\[
   \bar y= \arg\min \psi(y), \quad
   \bar z = \textup{Prox}_{p^*/\sigma}(\tilde x/\sigma - \cA^*\bar y + c ).
\]
Note that $\psi(\cdot)$ is strongly convex and continuously differentiable on $\idch$ with
\[\nabla \psi(y) = \nabla h^*(y) - \cA\,\textup{Prox}_{\sigma p}(\tilde x - \sigma(\cA^*y - c)),\quad\forall y\in\idch.\]
Thus, $\bar y$ can be obtained via solving the following nonsmooth equation
\begin{equation}\label{eq-xi}
\nabla \psi(y) = 0, \quad y\in\idch.
\end{equation}
Let $y\in\idch$ be any given point. Since $h^*$ is a convex function with a locally Lipschitz
continuous gradient on $\idch$, the following operator is well defined:
\[\hat\partial^2 \psi(y) := \partial( \nabla h^*)(y) + \sigma \cA \partial\textup{Prox}_{\sigma p}(\tilde x - \sigma (\cA^*y - c))\cA^*, \]
where $\partial( \nabla h^*)(y)$ is the Clarke subdifferential of $\nabla h^*$ at $y$ \cite{Clarke83}, and $\partial\textup{Prox}_{\sigma p}(\tilde x - \sigma (\cA^*y - c))$ is the Clarke subdifferential of the Lipschitz continuous mapping $\textup{Prox}_{\sigma p}(\cdot)$
at $\tilde x - \sigma (\cA^*y - c)$.
Note that from \cite[Proposition 2.3.3 and Theorem 2.6.6]{Clarke83}, we know that
\[
{\partial}^2 \psi(y)\, (d) \subseteq \hat{\partial}^2 \psi(y)\, (d), \quad  \forall \, d \in \cY,
\]
where ${\partial}^2 \psi(y)$ denotes the generalized Hessian of $\psi$ at $y$.
Define
\begin{equation}\label{p2:eq-netwon-partial}
V  := H  + \sigma \cA U \cA^*
 \end{equation}
with $H \in \partial^2 h^*(y)$ and $U  \in \partial\textup{Prox}_{\sigma p}(\tilde x - \sigma (\cA^* y - c)). $
 Then, we have $V \in\hat\partial^2\psi(y)$. Since $h^*$ is a strongly convex function, we know that $H $ is symmetric positive definite on $\cY$ and thus $V $ is also symmetric positive definite on $\cY$.

{Under the mild assumption that
$\nabla h^*$ and $\textup{Prox}_{\sigma p}$ are strongly semismooth (whose definition is given next),
we can design a superlinearly convergent semismooth Newton method to solve the
nonsmooth equation \eqref{eq-xi}.
}

\begin{definition}[Semismoothness \cite{Mifflin77,QiSun93,SunS2002}]
  Let $F:\cO \subseteq \cX \rightarrow \cY$ be a locally Lipschitz continuous function on the open set $\cO$. $F$ is said to be semismooth at $x\in \cO$ if
  $F$ is directionally differentiable at $x$ and for any $V \in \partial F(x + \Delta x)$ with $\Delta x\rightarrow 0$,
        \[F(x+\Delta x) - F(x) - V\Delta x = o(\norm{\Delta x}).\] $F$ is said to be strongly semismooth at $x$ if F is semismooth at $x$ and
        \[F(x+\Delta x) - F(x) - V\Delta x = O(\norm{\Delta x}^2).\]
    $F$ is said to be a semismooth (respectively, strongly semismooth) function on $\cO$
if it is semismooth (respectively, strongly semismooth) everywhere in $\cO$.
\end{definition}
Note that it is widely known in the nonsmooth optimization/equation community that   continuous piecewise affine functions and twice continuously differentiable functions are {all} strongly semismooth everywhere.  In particular, $\textup{Prox}_{\norm{\cdot}_1}$, as a Lipschitz continuous piecewise affine function, is strongly semismooth. See \cite{facchinei2003finite} for more semismooth and strongly semismooth functions.

Now, we can design a semismooth Newton ({\sc Ssn}) method to solve (\ref{eq-xi}) as follows and could expect to get a fast superlinear or even quadratic convergence.

\bigskip

\centerline{\fbox{\parbox{\textwidth}{
{\bf Algorithm {\sc Ssn}}: {\bf A semismooth Newton algorithm for solving (\ref{eq-xi})} ({\sc Ssn}($y^0, \tilde x, \sigma$)).
\\[5pt]
Given $\mu \in (0, 1/2)$, $\bar{\eta} \in (0, 1)$, $\tau \in (0,1]$, and $\delta \in (0, 1)$. Choose $y^0\in\textup{int}(\textup{dom}\,h^*)$.  Iterate the following steps for $j=0,1,\ldots.$
\begin{description}
\item[Step 1.]  Choose $H_j\in \partial (\nabla h^*)(y^j)$ and $U_j\in \partial\textup{Prox}_{\sigma p}(\tilde x - \sigma (\cA^* y^j - c))$.    Let $V_j := H_j + \sigma \cA U_j \cA^*$.
Solve the following linear system
\begin{equation}\label{eqn-epsk}
V_j d + \nabla \psi(y^j) = 0
\end{equation}
exactly or by the conjugate gradient (CG) algorithm
to find $d^j$
such that
\[
\norm{V_j d^j  + \nabla \psi(y^j)}\le \min(\bar{\eta}, \| \nabla \psi(y^j)\|^{1+\tau}).
\]
\item[Step 2.] (Line search)  Set $\alpha_j = \delta^{m_j}$, where $m_j$ is the first nonnegative integer $m$ for which
                         \begin{eqnarray}
                         y^j + \delta^{m} d^j \in \idch\quad \textup{and}\quad
                          \psi(y^j + \delta^{m} d^j) \leq \psi(y^j) + \mu \delta^{m}
                           \langle \nabla \psi(y^j), d^j \rangle. \nn
                          \end{eqnarray}
\item[Step 3.] Set $y^{j+1} = y^j + \alpha_j \, d^j$.
\end{description}
}}}

\medskip

The convergence results for the above {\sc Ssn} algorithm are stated in the next theorem.
\begin{theorem}\label{convergence-zwy-newton}
 Assume that $\nabla h^*(\cdot)$ and $\textup{Prox}_{\sigma p}(\cdot)$ are strongly semismooth on $\idch$ and $\cX$, respectively. Let the sequence $\{y^j\}$ be generated by Algorithm {\sc Ssn}.
Then $\{y^j\}$ converges to the unique optimal solution $\bar y \in \idch$ of the  problem in (\ref{eq-xi}) and
\[
\|y^{j+1} - \bar y \| = O(\| y^{j} - \bar y \| ^{1+\tau}). \nn
\]
\end{theorem}
\begin{proof}
 Since, by \cite[Proposition 3.3]{SDPNAL}, $d^j$ is a descent direction, Algorithm \SSNCG is well-defined.  By (\ref{p2:eq-netwon-partial}), we know that for any $j\ge 0$, $V_j \in \hat{\partial}^2 \psi(y^j)$.
 Then, we can prove the {conclusion} of this theorem by following the proofs to
 \cite[Theorems 3.4 and 3.5]{SDPNAL}. We omit the details here.
% , $\nabla \psi$ is also strongly semismooth on $\idch$.
 \end{proof}

We shall now discuss the implementations of stopping criteria (A), (B1) and (B2)
{for} Algorithm \SSNCG to solve the  subproblem (\ref{p2:alm-sub}) in Algorithm {\sc Ssnal}.
Note that when  \SSNCG is applied to minimize $\Psi_k(\cdot)$ to find
\[y^{k+1} = \textup{{\sc Ssn}}(y^k ,x^k,\sigma_k)\quad \textup{and}\quad
z^{k+1}=\textup{Prox}_{p^*/\sigma_k}(  x^k/\sigma_k - \cA^* y^{k+1} + c ),\] we have, by simple calculations and the strong convexity of $h^*$, that
\[ \Psi_k(y^{k+1},z^{k+1}) - \inf \Psi_k = \psi_k(y^{k+1}) - \inf \psi_k \le (1/2\alpha_h) \norm{\nabla \psi_k(y^{k+1})}^2  \]
and 
$(\nabla \psi_k(y^{k+1}), 0) \in \partial \Psi_k(y^{k+1},z^{k+1})$, where $\psi_k(y): = \inf_z\Psi_k(y,z)$ for all $y\in\cY$.
Therefore, the stopping criteria (A), (B1) and (B2) can be
{achieved} by the following implementable criteria
\begin{eqnarray}
& ({\rm A}')\quad \norm{\nabla{\psi_k(y^{k+1})}} \le \sqrt{\alpha_h/ \sigma_k}\, \varepsilon_k ,\quad \sum_{k=0}^{\infty} \varepsilon_k < \infty, \nn \\[5pt]
& ({\rm B1}')\quad\norm{\nabla{\psi_k(y^{k+1})}} \le \sqrt{\alpha_h \sigma_k}\, \delta_k \norm{\cA^*y^{k+1} + z^{k+1} - c}, \quad
\sum_{k=0}^\infty\delta_k < +\infty, \nn \\[5pt]
&({\rm B2}')\quad\norm{\nabla{\psi_k(y^{k+1})}} \le \delta_k' \norm{\cA^*y^{k+1} + z^{k+1} - c}, \quad
0 \le \delta_k' \to 0. \nn
\end{eqnarray}
That is, the stopping criteria (A), (B1) and (B2) will be satisfied as long as $\norm{\nabla \psi_k(y^{k+1})}$ is sufficiently small.

\subsection{An efficient implementation of \SSNCG for solving subproblems \eqref{p2:alm-sub}}
When Algorithm {\sc Ssnal} is applied to solve the general convex composite optimization model $({\bf P})$, the key part is to use Algorithm {\sc Ssn} to solve the subproblems \eqref{p2:alm-sub}.
In this subsection, we shall discuss an efficient implementation of \SSNCG for solving the aforementioned subproblems when the nonsmooth regularizer $p$ is chosen to be $\lambda\norm{\cdot}_1$ for some $\lambda >0$.
Clearly, in Algorithm {\sc Ssn}, the most important step is the computation of the search direction $d^j$
%at $y^j$
 from the linear system \eqref{eqn-epsk}.
 So we shall first discuss the solving of  this linear system.

Let  $(\tilde x,y)\in\Re^n\times\Re^m$ and $\sigma >0$ be given. We consider the following Newton linear system
\begin{equation}\label{eq:general_NTeq}
(H + \sigma  A  U  A^T) d = -\nabla \psi(y),
\end{equation}
where $H\in\partial (\nabla h^*)(y)$, $A$ denotes the matrix representation of $\cA$ with respect to the standard bases of $\Re^n$ and $\Re^m$, $U \in \partial{\rm Prox}_{\sigma \lambda \norm{\cdot}_1}( x)$ with $  x := \tilde x - \sigma  (A^T y -c)$. Since
$H $ is a symmetric and positive definite matrix, equation \eqref{eq:general_NTeq} can be equivalently rewritten as
\[\big(I_{m} + \sigma (L^{-1} A) U (L^{-1}A)^T\big)( L^T d) = -L^{-1} \nabla \psi(y),\]
where $L$ is a nonsingular  matrix obtained from the (sparse) Cholesky decomposition of $H$ such that $H = LL^T$. In many applications, $H$ is usually a sparse matrix. Indeed, when the function $h$ in the primal objective is taken to be the squared loss or the logistic loss functions, the %resulted
{resulting}
matrices $H$ are in fact diagonal matrices. That is, the costs of computing
$L$ and its inverse are negligible in most situations.
Therefore, without loss of generality, we can consider a
simplified version of \eqref{eq:general_NTeq} as follows
\begin{equation}\label{eq:LassoNTeq}
(I_m + \sigma  A  U  A^T) d = -\nabla \psi(y),
\end{equation}
which is precisely the Newton system associated with the standard Lasso problem \eqref{eq-l1}.
Since $U\in\Re^{n\times n}$ is a diagonal matrix, at the first glance, the costs of computing $A U A^T$ and the  matrix-vector multiplication $A U A^T d$ for a given vector $d\in \Re^m$ are $\cO(m^2n)$ and $\cO(mn)$, respectively.
%, which are the same as needed to compute the matrix $AA^T$ and .
These computational costs are too expensive  when the dimensions of $A$ are large and can make the commonly
%used
{employed}
approaches such as the Cholesky factorization and the conjugate gradient method inappropriate for solving
\eqref{eq:LassoNTeq}.
Fortunately, under the sparse optimization setting, if the sparsity of $U$ is wisely taken into the consideration, one can  substantially reduce  these unfavorable computational costs to a level such that they are  negligible or at least insignificant  compared to other costs. Next,
we shall show how this can be done by taking  full advantage of the sparsity  of $U$. This sparsity will be referred as the second order sparsity of the underlying problem.

{For $ x = \tilde x - \sigma ( A^T y - c)$,  in our computations, we can always choose 
$U = {\rm Diag}(u)$, the diagonal matrix whose
$i$th diagonal element is given by $u_i$
with
\begin{equation*}
u_i = \left\{\begin{aligned}
& 0 ,\quad \textup{if} \quad |  {x_i}|\le \sigma\lambda,
\\[5pt]
&1 , \quad \textup{otherwise},
\end{aligned}
\quad i=1,\ldots,n.
\right.
\end{equation*}
Since ${\rm Prox}_{\sigma \lambda \norm{\cdot}_1}(x) = {\rm sign}(x)\circ \max\{|x| - \sigma\lambda, 0\}$, it is not difficult to see that $U\in\partial{\rm Prox}_{\sigma \lambda \norm{\cdot}_1}(x) $.
}
Let $\cJ:=\{j \mid |{x}_j|> \sigma\lambda,\, j=1,\ldots,n \}$ and {$r=|\cJ|$, the cardinality of
$\cJ$.}
By taking the special $0$-$1$ structure of $U$ into consideration, we have {that}
\begin{equation}\label{eq:AUAT}
AUA^T = (AU)(AU)^T = A_{\cJ} A_{\cJ}^T,
\end{equation}
where $A_{\cJ}\in \Re^{m\times r}$ is the sub-matrix of $A$ with  those columns not in  $\cJ$ being removed from $A$. Then, by using \eqref{eq:AUAT}, we know that now the costs of computing $AUA^T$ and $A U A^T d$ for a given vector $d$ are reduced to $\cO(m^2 r)$ and $\cO(m r)$, respectively.  Due to the sparsity promoting property of the regularizer $p$, the number $r$ is usually much smaller than $n$. Thus, by     exploring   the aforementioned second order sparsity, we can greatly reduce the computational costs in solving the linear system \eqref{eq:LassoNTeq}
{when}
 using  the Cholesky factorization. % or the conjugate gradient method.
 More specifically, the total computational costs of using the Cholesky factorization to solve the linear system  are reduced from $\cO(m^2 (m + n))$ to
$\cO(m^2 (m + r))$. See Figure \ref{fig:AAttoAuAt} for an illustration on the reduction.
This means that even if $n$ happens to be   extremely large (say, larger than  $10^7$), one can still solve the Newton linear system \eqref{eq:LassoNTeq} efficiently via the Cholesky factorization as long as both $m$ and $r$ are moderate (say, less than $10^4$).
\begin{figure}[htp]
\centering
	\begin{tikzpicture}[line width=1pt]
	\node[above] at (-0.3,0.5) {$m$};
	\node[above] at (1.5, 1) {$n$};
	\node[left] at (-0.5,0.3) {$ AUA^T=$};
	%\draw[red,dashed](1,0)--(1,1);
	\fill[blue] (0,0) rectangle (0.3,1);
	\fill[red]  (1,0) rectangle (5,1);
	\fill[blue] (5.2,1) rectangle (6.2,0.7);
	\fill[red] (5.2,0) rectangle (6.2,-4);
	\draw (0,0) rectangle (5,1);
	\draw (5.2,1) rectangle (6.2,-4);
	\draw[vecArrow] (2.5,-2) to (2.5,-4);
	\node[right] at (6.3,0) { $\cO(m^2n)$};
	\end{tikzpicture}
 	\begin{tikzpicture}[line width=1pt]
 	\node[right] at (5.6,0.4) {$ A_{\cJ} A_{\cJ}^T=$};
 	\node[above] at (7.8,0.5) {$m$};
 	\node[above] at (8.1, 1) {$r$};
 	\draw  (8,0) rectangle (9,1);
 	\fill[blue] (8,0) rectangle (8.3,1);
 	\draw  (9.2,0) rectangle (10.2,1);
 	\fill[blue] (9.2,1) rectangle (10.2,0.7);
 	\node[right] at (10.2,0.4) {$=$};
 	\fill[blue] (10.8,0) rectangle (11.8,1);
 	\draw (10.8,0) rectangle (11.8,1);
 	\node[right] at (12,0.5) { $\cO(m^2r)$};
 	\end{tikzpicture}
\caption{Reducing the computational costs from $\cO(m^2 n)$ to $\cO(m^2 r)$}
\label{fig:AAttoAuAt}
\end{figure}
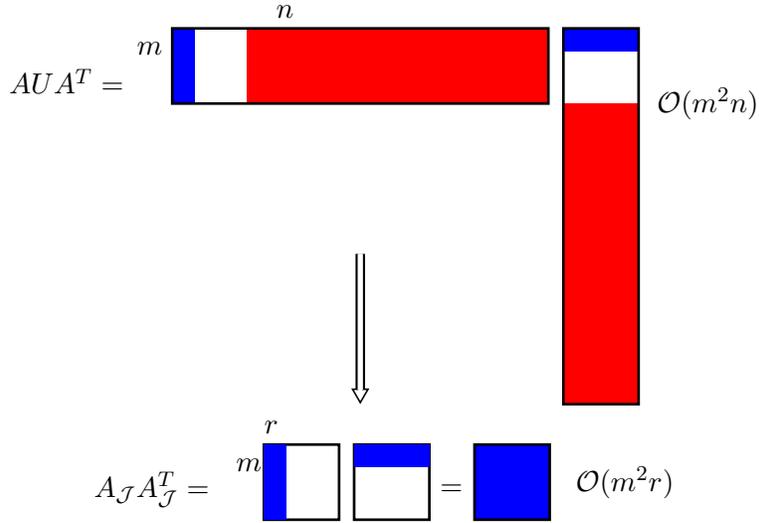
If, in addition, $r \ll m$, which is often the case when $m$ is large and the optimal solutions to the underlying problem are sparse, instead of factorizing an $m\times m$ matrix, we can make use of the Sherman-Morrison-Woodbury formula \cite{matrixcom1996} to get the inverse of $I_m + \sigma A U A^T$ by inverting a much smaller $r \times r  $ matrix as follows:
\[(I_m + \sigma A U A^T)^{-1}
= (I_m + \sigma A_{\cJ}A_{\cJ}^T)^{-1}
= I_m - A_\cJ(\sigma^{-1} I_{r} + A_{\cJ}^T A_{\cJ})^{-1}A_{\cJ}^T.\]
 See  Figure \ref{fig:AjtAj} for an illustration on %computing
{the computation of}
 $A_{\cJ}^T A_{\cJ}$. In this case, the total computational costs for solving the Newton linear  system \eqref{eq:LassoNTeq} are  reduced significantly further from $\cO(m^2 (m + r))$ to $\cO(r^2 (m + r))$.
We %shall
{should}
emphasize here that this dramatic  reduction on the computational costs results from the wise combination of the careful examination of the  existing second order sparsity in the Lasso-type  problems and some ``smart'' numerical linear algebra.
\begin{figure}[htb]
	\begin{center}
		\begin{tikzpicture}[line width=1pt]
		\fill[blue] (0,0) rectangle (1,0.3);
		\draw (0,0) rectangle (1,0.3);
		\node[left] at (0,0.15) {$r$};
        \node[above] at (0.4,0.3) {$m$};
		\fill[blue] (1.2,0.3) rectangle (1.5, -0.7);
		\draw (1.2,0.3) rectangle (1.5, -0.7);
		\node[right] at (1.6,-0.3) {$=$};
		\fill[blue] (2.3,-0.4) rectangle (2.6,-0.1);
		\draw (2.3,-0.4) rectangle (2.6,-0.1);
		\node[left] at (-0.2,-0.3) {$A_{\cJ}^T A_{\cJ} =$};
		\node[right] at (2.8,-0.3) {$\cO(r^2m $)};
		\end{tikzpicture}
	\end{center}
   \caption{Further reducing the computational costs to $\cO(r^2 m)$}
   \label{fig:AjtAj}
\end{figure}

 From the above arguments, we can see that as long as the number of the nonzero components of $\textup{Prox}_{\sigma \lambda \norm{\cdot}_1}(x)$ is small, %e.g.,
{say}
 less than $\sqrt{n}$ and $H_j\in \partial (\nabla h^*)(y^j)$ is a sparse matrix, e.g., a diagonal matrix, we can always solve  the linear system (\ref{eqn-epsk}) at very low costs. In particular, this is true for the Lasso problems admitting  sparse solutions. Similar discussions on the reduction of {the} computational costs can also be conducted
%to
{for}
the case when the conjugate gradient method is applied to solve the linear systems (\ref{eqn-epsk}). Note that one may argue that even if the original problem has only sparse solutions, at certain stages, one may still encounter the situation that the number of the nonzero components of $\textup{Prox}_{\sigma \lambda \norm{\cdot}_1}(x)$ is large. Our answer to this question is   simple. Firstly, this phenomenon  {rarely occurs}
 in practice since we always start with a sparse feasible point, e.g., %the vector of all zeros.
{the zero vector.}
Secondly, even at certain steps this phenomenon does occur, we %can
just {need to} apply a small number of conjugate gradient iterations to the linear system (\ref{eqn-epsk}) as in this case the parameter $\sigma$ is normally small and the current point is far away from any sparse optimal solution.
In summary, we have demonstrated how Algorithm {\sc Ssn} can be implemented efficiently for solving  sparse optimization problems of the form \eqref{p2:alm-sub} with $p(\cdot)$ being chosen to be $\lambda\norm{\cdot}_1$.

\section{Numerical experiments for Lasso problems}
In this section, we shall evaluate the performance {of}
our algorithm {\sc Ssnal} for solving large scale Lasso problems \eqref{eq-l1}.
We note that the relative performance of most of the existing algorithms mentioned in the introduction has recently been well
documented in the two recent papers \cite{Fountoulakis14matrix,Milzarek14}, which  appears to
suggest that for some large scale sparse reconstruction problems, mfIPM\footnote{\url{http:
		//www.maths.ed.ac.uk/ERGO/mfipmcs/}}
and FPC\_AS\footnote{\url{http://www.caam.rice.edu/~optimization/L1/FPC_AS/}} have mostly outperformed the other solvers.  Hence, in this section
we will compare our algorithm with these two popular solvers.
{Note that mfIPM is a specialized  interior-point based second-order
method  designed for the Lasso problem \eqref{eq-l1},
whereas FPC\_AS is a first-order method based on forward-backward operator splitting.}
Moreover, we also report the numerical performance of two commonly used algorithms for solving Lasso problems: {the accelerated proximal gradient (APG) algorithm as} implemented by Liu et al. in SLEP\footnote{\url{http://yelab.net/software/SLEP/}} \cite{Liu2009SLEP} and the alternating direction method of multipliers (ADMM) \cite{ADMM1,ADMM2}.
For  the purpose of
comparisons, we also test the linearized ADMM (LADMM) \cite{LADMM}.
We have implemented both ADMM and LADMM in {\sc Matlab} with the step-length set to be
{1.618.}
%the golden ratio.
Although the %available current
{existing}
solvers can perform impressively well on some easy-to-solve sparse reconstruction problems,
{as one will see later,}
they  lack the ability %in efficiently  solving
{to efficiently solve}
difficult    problems such as the large scale regression problems when the data $\cA$ is badly conditioned.

For  the
testing purpose, the regularization parameter $\lambda$ in {the} Lasso problem (\ref{eq-l1}) is chosen as \[\lambda = \lambda_{c}\norm{\cA^*b}_{\infty},\] where $0<\lambda_c<1$.
In our numerical experiments, we measure the accuracy of an approximate optimal solution $\tilde x$ for (\ref{eq-l1}) by using the following relative KKT residual:
\[\eta = \frac{\norm{\tilde x - \textup{prox}_{\lambda\norm{\cdot}_1}(\tilde x - \cA^*(\cA \tilde x - b))}}{1 + \norm{\tilde x} + \norm{\cA \tilde x -b}}.\]
For a given tolerance $\epsilon >0$, we will stop the tested algorithms when $\eta < \epsilon$. For all the tests in this section, we set $\epsilon = 10^{-6}$. The algorithms will also be stopped when they
reach the maximum number of iterations (1000 iterations for our algorithm and mfIPM, and 20000 iterations for FPC\_AS, APG, ADMM and LADMM) or the maximum computation time {of 7 hours}. All the parameters for mfIPM, FPC\_AS and APG are set to the default values. All our computational results  are obtained by running {\sc Matlab} (version 8.4) on a windows workstation (16-core, Intel Xeon E5-2650 @ 2.60GHz, 64 G RAM).

\subsection{Numerical results for large scale regression problems}

In this subsection, we test all the algorithms with the test instances $(\cA, b)$  obtained from large scale regression problems in the LIBSVM datasets \cite{Chang2011a}. These data sets
are collected from 10-K Corpus \cite{Kogan2009predicting} and UCI data repository \cite{Lichman2013UCI}. As suggested in
\cite{Huang2010predicting}, for the data sets {\bf pyrim, triazines, abalone, bodyfat, housing, mpg, space\_ga}, we expand their original features by using polynomial basis functions over those features. For example, the last digits in {\bf pyrim5 }
indicates that an order 5 polynomial is used to generate the basis functions. This naming convention is also used in the rest of the expanded data sets. These test instances, shown in Table \ref{table:UCIsta}, can be quite difficult in terms of the problem dimensions and the largest eigenvalues of $\cA\cA^*$, which is denoted as $\lambda_{\max}(\cA\cA^*)$.

\begin{scriptsize}
	\begin{longtable}{| c |c|c|}
		\caption{Statistics of the UCI test instances.}\label{table:UCIsta}
		\\
		\hline
		\mc{1}{|c|}{ }  &\mc{1}{c|}{ } &\mc{1}{c|}{ }\\[-5pt]
		\mc{1}{|c|}{probname}  &\mc{1}{c|}{$m;n$} &\mc{1}{c|}{$\lambda_{\max}(\cA\cA^*)$} \\[2pt]
		\hline
		\endhead
     E2006.train
     & 16087$;$150360& 1.91e+05\\[2pt]
     \hline
     log1p.E2006.train
     & 16087$;$4272227& 5.86e+07\\[2pt]
     \hline
     E2006.test
     & 3308$;$150358& 4.79e+04\\[2pt]
     \hline
     log1p.E2006.test
     & 3308$;$4272226& 1.46e+07\\[2pt]
     \hline
     pyrim5
     & 74$;$201376& 1.22e+06\\[2pt]
     \hline
     triazines4
     & 186$;$635376& 2.07e+07\\[2pt]
     \hline
     abalone7
     & 4177$;$6435& 5.21e+05\\[2pt]
     \hline
     bodyfat7
     & 252$;$116280& 5.29e+04\\[2pt]
     \hline
     housing7
     & 506$;$77520& 3.28e+05\\[2pt]
     \hline
     mpg7
     & 392$;$3432& 1.28e+04\\[2pt]
     \hline
     space\_ga9
     & 3107$;$5005& 4.01e+03\\[2pt]
     \hline
	\end{longtable}
\end{scriptsize}

Table \ref{table:UCI} reports the detailed numerical results for {\sc Ssnal}, mfIPM, FPC\_AS, APG, LADMM and ADMM in solving  large scale regression problems. In the table, $m$ denotes the {number of samples}, %dimension of samples,
$n$ denotes the number of features and
``nnz'' denotes the number of nonzeros in the solution $x$ obtained by {\sc Ssnal}
using the following estimation
\[{\rm nnz}:= \min \Big\{k \mid \sum_{i=1}^k |\hat x_i| \ge 0.999\norm{x}_1\Big\},\]
where $\hat x$ is obtained by sorting $x$  {such that $|\hat{x}_1| \geq \ldots \geq |\hat{x}_n|$.}
%in the decreasing order of absolute values of its components.
One can observe from Table \ref{table:UCI} that all the tested first order algorithms except ADMM, i.e., FPC\_AS, APG and LADMM
fail to solve most of the test instances to the required accuracy after 20000 iterations or 7 hours.
In particular, FPC\_AS fails to produce a reasonably accurate solution for all the test instances.
In fact, for 3 test instances, it breaks down due to some internal errors.
This poor performance indicates that these first order methods cannot obtain reasonably accurate solutions
%when used to deal
{when dealing}
with difficult large scale problems.
{While ADMM can solve most of the test instances, it needs much more time than {\sc Ssnal}.} For example,
%considering problem
{for the instance}
{\bf housing7}
with $\lambda_c = 10^{-3}$, we can see that {\sc Ssnal} is at least 330 times faster than ADMM. In addition,
{\sc Ssnal} can solve %problem
{the instance}
 {\bf  pyrim5} in 9 seconds while ADMM %runs
{reaches the maximum of
20000 iterations and consumes about 2 hours but only produces a rather} inaccurate solution.

On the other hand, one can observe that the two second order information based methods {\sc Ssnal} and
mfIPM perform quite {robustly} despite the huge dimensions and the possibly badly conditioned data sets. More specifically, {\sc Ssnal} is able to solve {the instance}
{\bf log1p.E2006.train}
with approximately 4.3 million features in 20 seconds ($\lambda_c = 10^{-3}$).
Among these two algorithms, clearly, \SSNAL is far more efficient than
{the specialized interior-point method} mfIPM for all the test instances, especially for
 large scale problems where the factor can be up to 300 times {faster}.
While \SSNAL can solve all the instances to the desired accuracy, as the %problem %getting progressively
{problems get progressively}
 more difficult ($\lambda_c$ decreases from $10^{-3}$ to $10^{-4}$), mfIPM fails on more test instances (2 out of 11 vs. 4 out of 11 instances). We also note that mfIPM can only reach a solution with the accuracy of $10^{-1}$ when it fails to compute the corresponding Newton directions. These facts indicate that the nonsmooth approach {employed by
\SSNAL
is far more superior compared to}
the interior point method in exploiting the  sparsity
{in the generalized Hessian.}
The superior numerical performance of \SSNAL indicates that {it} is a robust, high-performance solver for high-dimensional Lasso problems.

{As pointed out by one referee, the polynomial expansion in our data processing step may affect the scaling of the problems. 
Since first order solvers are not affine invariant, this scaling may affect their performance. Hence,  we normalize the matrix $A$ (the matrix representation of $\cA$) to have columns with at most unit norm and 
correspondingly change the variables.  This scaling step also changes the regularization parameter $\lambda$ accordingly to a nonuniform weight vector. Since it is not easy to call FPC\_AS when $\lambda$ is not a scalar, 
based on the recommendation of the referee, 
we use another popular active-set based solver  PSSas\footnote{https://www.cs.ubc.ca/~schmidtm/Software/thesis.html}
\cite{Schmidt2010thesis} to replace FPC\_AS for the testing. All the parameters for PSSas are set to the default values. In our tests, PSSas will be stopped when it reaches the maximum number of 20000 iterations or the maximum computation time of 7 hours. We note that
by default, PSSas will terminate when its progress is smaller than the threshold $10^{-9}$.

The detailed numerical results for {\sc Ssnal}, mfIPM, PSSas, APG, LADMM and ADMM  with the normalization step in solving the large scale regression problems are listed in Table \ref{table:UCI_scale}.
From Table \ref{table:UCI_scale}, one can easily observe that the simple normalization technique does not change the conclusions based on Table \ref{table:UCI}. The performance of \SSNAL is generally invariant with respect to the scaling and \SSNAL is still much faster and more robust than  other solvers.
Meanwhile, after the normalization, mfIPM now can solve 3 more instances to the required accuracy. On the other hand, APG and the ADMM type of solvers (i.e., LADMM and ADMM) perform worse than the un-scaled case. Besides, PSSas can only solve 5 out of 22 instances to the required accuracy.
In fact, PSSas fails on all the test instances when $\lambda_c = 10^{-4}$. For the instance {\bf triazines4}, it consumes about 6 hours but only generates a poor solution with $\eta = 2.3\times 10^{-2}$. (Actually, we also run PSSas on these test instances without scaling and obtain similar performance. Detailed results are omitted to conserve space.)
Therefore, we can safely conclude that the simple normalization technique employed here may not be suitable for general first order methods.}

\begin{scriptsize}
	\begin{longtable}{| @{}l@{} |@{}c@{}|@{}c@{}| c| r|}
		\caption{The performance of {\sc Ssnal}, mfIPM, FPC\_AS, APG, LADMM and ADMM on 11 selected regression problems
			(accuracy $\epsilon= 10^{-6}$). $m$ is the sample size and $n$ is the dimension of features. In the table, ``a'' = {\sc Ssnal}, ``b'' =  mfIPM, ``c'' = FPC\_AS, ``d'' =  APG,
			``e'' = LADMM, and ``f'' = ADMM, respectively. ``nnz'' denotes the number of nonzeros in the solution obtained by {\sc Ssnal}. ``{\bf Error}'' indicates the algorithm breaks down due to some internal errors. The computation time is in the format of ``hours:minutes:seconds''. ``00'' in the time column means less than 0.5 seconds.}\label{table:UCI}
		\\
		\hline
		\mc{1}{|c|}{} &\mc{1}{c|}{} &\mc{1}{c|}{} &\mc{1}{c|}{}&\mc{1}{c|}{}\\[-5pt]
		\mc{1}{|c|}{} & \mc{1}{c|}{$\lambda_c$} &\mc{1}{c|}{nnz} &\mc{1}{c|}{$\eta$} &\mc{1}{c|}{time} \\[2pt] \hline
		\mc{1}{|c|}{probname}  &\mc{1}{c|}{} &\mc{1}{c|}{}
		&\mc{1}{c|}{a $|$ b $|$ c $|$ d $|$ e $|$ f} &\mc{1}{c|}{a $|$ b $|$ c $|$ d $|$ e $|$ f} \\[2pt]
		\mc{1}{|c|}{$m;n$} &\mc{1}{c|}{ } &\mc{1}{c|}{}
		&\mc{1}{c|}{ } &\mc{1}{c|}{ } \\ \hline
		\endhead

E2006.train
& $10^{-3}$ &  1 	 &   1.3-7 $|$    3.9-7 $|$    {\red{9.0-4}} $|$    1.1-8 $|$    5.1-14 $|$    9.1-7 	 &01 $|$ 11 $|$   1:34:40 $|$ 01 $|$ 03 $|$ 11:00\\[2pt]
16087$;$150360& $10^{-4}$ &  1 	 &   3.7-7 $|$    1.6-9 $|$    {\red{9.7-4}} $|$    1.5-7 $|$    2.6-14 $|$    7.3-7 	 &01 $|$ 14 $|$   1:37:52 $|$ 01 $|$ 04 $|$ 11:23\\[2pt]
\hline
log1p.E2006.train
& $10^{-3}$ &  5 	 &   5.6-7 $|$    4.7-8 $|$    {\red{4.7-1}} $|$    {\red{ 1.1-5}} $|$    {\red{6.4-4}} $|$    9.9-7 	 &20 $|$ 41:01 $|$   7:00:20 $|$   2:00:33 $|$   2:22:21 $|$ 39:36\\[2pt]
16087$;$4272227& $10^{-4}$ & 599 	 &   4.4-7 $|$    {\red{7.9-1}} $|$    {\red{1.6-1}} $|$    {\red{1.5-4}} $|$    {\red{6.3-3}} $|$    9.8-7 	 &55 $|$   3:20:44 $|$   7:00:00 $|$   2:04:54 $|$   2:24:23 $|$ 36:09\\[2pt]
\hline
E2006.test
& $10^{-3}$ &  1 	 &   1.6-9 $|$    2.9-7 $|$    {\red{3.7-4}} $|$    5.5-8 $|$    4.0-14 $|$    6.7-7 	 &00 $|$ 05 $|$ 33:30 $|$ 00 $|$ 01 $|$ 26\\[2pt]
3308$;$150358& $10^{-4}$ &  1 	 &   2.1-10 $|$    2.9-7 $|$    {\red{4.3-4}} $|$    3.7-7 $|$    2.9-10 $|$    6.3-7 	 &00 $|$ 05 $|$ 32:36 $|$ 00 $|$ 01 $|$ 27\\[2pt]
\hline
log1p.E2006.test
& $10^{-3}$ &  8 	 &   9.2-7 $|$    1.5-8 $|$    {\red{9.8-1}} $|$    {\red{1.1-4}} $|$    9.9-7 $|$    9.9-7 	 &17 $|$ 37:24 $|$   7:00:01 $|$   1:25:11 $|$ 22:52 $|$ 5:58\\[2pt]
3308$;$4272226& $10^{-4}$ & 1081 	 &   2.2-7 $|$    {\red{8.7-1}} $|$    {\red{8.7-1}} $|$    {\red{3.9-4}} $|$    9.9-7 $|$    9.9-7 	 &34 $|$   7:00:24 $|$   7:00:01 $|$   1:23:07 $|$   1:00:10 $|$ 4:00\\[2pt]
\hline
pyrim5
& $10^{-3}$ & 72 	 &   9.9-7 $|$    3.2-7 $|$    {\red{8.8-1}} $|$    {\red{8.9-4}} $|$    {\red{4.0-4}} $|$    {\red{ 3.2-5}} 	 &05 $|$ 21:27 $|$   2:01:23 $|$ 8:48 $|$ 9:30 $|$ 23:17\\[2pt]
74$;$201376& $10^{-4}$ & 78 	 &   7.1-7 $|$    6.6-8 $|$    {\red{9.8-1}} $|$    {\red{3.4-3}} $|$    {\red{3.7-3}} $|$    {\red{1.6-3}} 	 &09 $|$ 49:49 $|$   1:09:16 $|$ 9:28 $|$ 9:55 $|$ 24:30\\[2pt]
\hline
triazines4
& $10^{-3}$ & 519 	 &   8.4-7 $|$    {\red{8.8-1}} $|$    {\red{9.1-1}} $|$    {\red{1.9-3}} $|$    {\red{2.9-3}} $|$    {\red{3.2-4}} 	 &36 $|$ 53:30 $|$   7:30:17 $|$ 54:44 $|$   1:02:26 $|$   2:17:11\\[2pt]
186$;$635376& $10^{-4}$ & 260 	 &   9.9-7 $|$    {\red{8.9-1}} $|$    {\red{9.8-1}} $|$    {\red{1.1-2}} $|$    {\red{1.6-2}} $|$    {\red{1.1-3}} 	 &1:44 $|$ 53:33 $|$   7:00:00 $|$   1:03:56 $|$   1:01:26 $|$   2:00:18\\[2pt]
\hline
abalone7
& $10^{-3}$ & 24 	 &   5.7-7 $|$    3.5-7 $|$ {\bf Error} $|$    {\red{ 5.3-5}} $|$    {\red{ 4.9-6}} $|$    9.9-7 	 &02 $|$ 49 $|$ {\bf Error} $|$ 10:38 $|$ 18:04 $|$ 9:52\\[2pt]
4177$;$6435& $10^{-4}$ & 59 	 &   3.7-7 $|$    4.7-7 $|$ {\bf Error} $|$    {\red{3.9-3}} $|$    {\red{ 3.3-5}} $|$    9.9-7 	 &03 $|$ 3:19 $|$ {\bf Error} $|$ 10:43 $|$ 13:58 $|$ 9:36\\[2pt]
\hline
bodyfat7
& $10^{-3}$ &  2 	 &   3.8-8 $|$    4.0-9 $|$    {\red{3.6-1}} $|$    8.8-7 $|$    9.0-7 $|$    9.9-7 	 &02 $|$ 1:29 $|$   1:12:02 $|$ 4:00 $|$ 3:08 $|$ 1:49\\[2pt]
252$;$116280& $10^{-4}$ &  3 	 &   4.6-8 $|$    3.4-7 $|$    {\red{1.9-1}} $|$    {\red{ 3.3-5}} $|$    9.8-7 $|$    9.9-7 	 &03 $|$ 2:41 $|$   1:13:08 $|$ 12:16 $|$ 4:19 $|$ 4:05\\[2pt]
\hline
housing7
& $10^{-3}$ & 158 	 &   2.3-7 $|$    {\red{8.1-1}} $|$    {\red{8.4-1}} $|$    {\red{2.6-4}} $|$    {\red{1.7-4}} $|$    9.9-7 	 &04 $|$   5:13:19 $|$   1:41:01 $|$ 16:52 $|$ 20:18 $|$ 22:12\\[2pt]
506$;$77520& $10^{-4}$ & 281 	 &   7.7-7 $|$    {\red{8.5-1}} $|$    {\red{4.2-1}} $|$    {\red{5.3-3}} $|$    {\red{5.9-4}} $|$    {\red{ 6.6-6}} 	 &08 $|$   5:00:09 $|$   1:39:36 $|$ 17:04 $|$ 20:53 $|$ 42:36\\[2pt]
\hline
mpg7
& $10^{-3}$ & 47 	 &   2.0-8 $|$    6.5-7 $|$ {\bf Error} $|$    {\red{ 1.9-6}} $|$    9.9-7 $|$    9.8-7 	 &00 $|$ 04 $|$ {\bf Error} $|$ 38 $|$ 14 $|$ 07\\[2pt]
392$;$3432& $10^{-4}$ & 128 	 &   4.1-7 $|$    8.5-7 $|$    {\red{4.5-1}} $|$    {\red{1.2-4}} $|$    {\red{ 1.3-6}} $|$    9.9-7 	 &00 $|$ 13 $|$ 4:09 $|$ 39 $|$ 1:01 $|$ 11\\[2pt]
\hline
space\_ga9
& $10^{-3}$ & 14 	 &   4.8-7 $|$    2.6-7 $|$    {\red{1.5-2}} $|$    2.5-7 $|$    9.9-7 $|$    9.9-7 	 &01 $|$ 16 $|$ 30:55 $|$ 1:46 $|$ 42 $|$ 37\\[2pt]
3107$;$5005& $10^{-4}$ & 38 	 &   3.7-7 $|$    3.0-7 $|$    {\red{4.0-2}} $|$    {\red{ 1.8-5}} $|$    9.9-7 $|$    9.9-7 	 &01 $|$ 41 $|$ 30:26 $|$ 6:46 $|$ 2:20 $|$ 56\\[2pt]
\hline

	\end{longtable}
\end{scriptsize}

		\begin{scriptsize}
	\begin{longtable}{| @{}l@{} |@{}c@{}|@{}c@{}| c|r|}
		\caption{The performance of {\sc Ssnal}, mfIPM, PSSas, APG, LADMM and ADMM on 11 selected regression problems with {\bf scaling}
			(accuracy $\epsilon= 10^{-6}$). In the table, ``a'' = {\sc Ssnal}, ``b'' =  mfIPM, ``c1'' = PSSas, ``d'' =  APG,
			``e'' = LADMM, and ``f'' = ADMM, respectively. ``nnz'' denotes the number of nonzeros in the solution obtained by {\sc Ssnal}.  The computation time is in the format of ``hours:minutes:seconds''. ``00'' in the time column means less than 0.5 seconds.}\label{table:UCI_scale}
		\\
		\hline
		\mc{1}{|c|}{} &\mc{1}{c|}{} &\mc{1}{c|}{} &\mc{1}{c|}{}&\mc{1}{c|}{}\\[-5pt]
		\mc{1}{|c|}{} & \mc{1}{c|}{$\lambda_c$} &\mc{1}{c|}{nnz} &\mc{1}{c|}{$\eta$} &\mc{1}{c|}{time} \\[2pt] \hline
		\mc{1}{|c|}{probname}  &\mc{1}{c|}{} &\mc{1}{c|}{}
		&\mc{1}{c|}{a $|$ b $|$ c1 $|$ d $|$ e $|$ f} &\mc{1}{c|}{a $|$ b $|$ c1 $|$ d $|$ e $|$ f} \\[2pt]
		\mc{1}{|c|}{$m;n$} &\mc{1}{c|}{ } &\mc{1}{c|}{}
		&\mc{1}{c|}{ } &\mc{1}{c|}{ } \\ \hline
		\endhead
		
		E2006.train 
		& $10^{-3}$ &  1 	 &   1.6-7 $|$    4.1-7 $|$    4.6-12 $|$    9.1-7 $|$    8.7-7 $|$    9.7-7 	 &01 $|$ 14 $|$ 01 $|$ 02 $|$ 05 $|$ 10:21\\[2pt] 
		16087$;$150360& $10^{-4}$ &  1 	 &   4.2-9 $|$    1.4-8 $|$    {\red{4.2-4}} $|$    3.6-7 $|$    5.9-7 $|$    9.6-7 	 &01 $|$ 16 $|$ 02 $|$ 03 $|$ 05 $|$ 13:52\\[2pt] 
		\hline 
		log1p.E2006.train 
		& $10^{-3}$ &  5 	 &   2.6-7 $|$    4.9-7 $|$    {\red{ 1.7-6}} $|$    {\red{1.7-4}} $|$    {\red{1.7-4}} $|$    {\red{ 2.8-5}} 	 &35 $|$ 59:55 $|$   1:29:14 $|$   2:17:57 $|$   3:05:04 $|$   7:00:01\\[2pt] 
		16087$;$4272227& $10^{-4}$ & 599 	 &   5.0-7 $|$    3.6-7 $|$    {\red{ 8.8-6}} $|$    {\red{1.1-2}} $|$    {\red{3.2-3}} $|$    {\red{1.0-4}} 	 &2:04 $|$   2:18:28 $|$   5:19:39 $|$   2:34:51 $|$   3:10:05 $|$   7:00:01\\[2pt] 
		\hline 
		E2006.test 
		& $10^{-3}$ &  1 	 &   1.6-7 $|$    1.3-7 $|$    6.5-10 $|$    3.9-7 $|$    2.4-7 $|$    9.9-7 	 &01 $|$ 08 $|$ 00 $|$ 01 $|$ 01 $|$ 28\\[2pt] 
		3308$;$150358& $10^{-4}$ &  1 	 &   3.2-9 $|$    2.5-7 $|$    {\red{8.9-4}} $|$    8.9-7 $|$    6.4-7 $|$    9.5-7 	 &01 $|$ 07 $|$ 01 $|$ 01 $|$ 02 $|$ 35\\[2pt] 
		\hline 
		log1p.E2006.test 
		& $10^{-3}$ &  8 	 &   1.4-7 $|$    9.2-8 $|$    {\red{ 1.5-6}} $|$    {\red{1.6-2}} $|$    {\red{1.9-4}} $|$    9.9-7 	 &27 $|$ 30:45 $|$   1:13:58 $|$   1:29:25 $|$   2:08:16 $|$   2:30:17\\[2pt] 
		3308$;$4272226& $10^{-4}$ & 1081 	 &   7.2-7 $|$    8.5-7 $|$    {\red{ 9.4-6}} $|$    {\red{3.7-3}} $|$    {\red{1.3-3}} $|$    {\red{2.9-4}} 	 &1:40 $|$   1:24:36 $|$   5:45:18 $|$   1:30:58 $|$   2:02:07 $|$   3:06:48\\[2pt] 
		\hline 
		pyrim5 
		& $10^{-3}$ & 70 	 &   2.5-7 $|$    4.2-7 $|$    {\red{7.2-3}} $|$    {\red{3.6-3}} $|$    {\red{1.0-3}} $|$    {\red{ 4.1-5}} 	 &05 $|$ 9:03 $|$ 16:42 $|$ 8:25 $|$ 10:12 $|$ 20:48\\[2pt] 
		74$;$201376& $10^{-4}$ & 78 	 &   4.6-7 $|$    7.7-7 $|$    {\red{1.3-2}} $|$    {\red{8.2-3}} $|$    {\red{3.7-3}} $|$    {\red{2.4-3}} 	 &06 $|$ 47:20 $|$ 34:03 $|$ 9:06 $|$ 10:48 $|$ 17:37\\[2pt] 
		\hline 
		triazines4 
		& $10^{-3}$ & 566 	 &   8.5-7 $|$    {\red{7.7-1}} $|$    {\red{2.0-3}} $|$    {\red{1.8-3}} $|$    {\red{1.3-3}} $|$    {\red{1.3-4}} 	 &29 $|$ 49:27 $|$   1:35:41 $|$ 55:31 $|$   1:06:28 $|$   2:28:23\\[2pt] 
		186$;$635376& $10^{-4}$ & 261 	 &   9.8-7 $|$    {\red{9.4-1}} $|$    {\red{2.3-2}} $|$    {\red{1.1-2}} $|$    {\red{2.6-2}} $|$    {\red{2.1-2}} 	 &1:14 $|$ 48:19 $|$   5:11:45 $|$   1:03:11 $|$   1:07:45 $|$   2:07:36\\[2pt] 
		\hline 
		abalone7 
		& $10^{-3}$ & 24 	 &   8.4-7 $|$    1.6-7 $|$    1.5-7 $|$    {\red{1.3-3}} $|$    {\red{1.5-4}} $|$    {\red{ 1.8-6}} 	 &02 $|$ 2:03 $|$ 1:59 $|$ 10:05 $|$ 11:54 $|$ 37:58\\[2pt] 
		4177$;$6435& $10^{-4}$ & 59 	 &   3.7-7 $|$    9.2-7 $|$    {\red{1.8-1}} $|$    {\red{7.3-2}} $|$    {\red{5.8-2}} $|$    9.9-7 	 &04 $|$ 9:47 $|$ 12:26 $|$ 10:30 $|$ 11:49 $|$ 14:19\\[2pt] 
		\hline 
		bodyfat7 
		& $10^{-3}$ &  2 	 &   1.2-8 $|$    5.2-7 $|$    {\red{ 2.1-5}} $|$    {\red{1.4-2}} $|$    {\red{8.7-2}} $|$    9.9-7 	 &02 $|$ 1:41 $|$ 3:28 $|$ 12:49 $|$ 15:27 $|$ 10:24\\[2pt] 
		252$;$116280& $10^{-4}$ &  3 	 &   6.4-8 $|$    7.8-7 $|$    {\red{2.4-1}} $|$    {\red{2.7-2}} $|$    {\red{9.0-2}} $|$    9.9-7 	 &03 $|$ 2:18 $|$ 6:14 $|$ 13:12 $|$ 15:13 $|$ 18:40\\[2pt] 
		\hline 
		housing7 
		& $10^{-3}$ & 158 	 &   8.8-7 $|$    6.6-7 $|$    9.9-7 $|$    {\red{4.1-4}} $|$    {\red{ 5.5-6}} $|$    9.9-7 	 &03 $|$ 6:26 $|$ 9:20 $|$ 17:00 $|$ 25:56 $|$ 20:06\\[2pt] 
		506$;$77520& $10^{-4}$ & 281 	 &   8.0-7 $|$    {\red{7.0-1}} $|$    {\red{ 1.4-5}} $|$    {\red{1.1-2}} $|$    {\red{1.9-3}} $|$    {\red{ 1.3-5}} 	 &05 $|$   4:33:23 $|$ 30:27 $|$ 17:34 $|$ 19:46 $|$ 49:20\\[2pt] 
		\hline 
		mpg7 
		& $10^{-3}$ & 47 	 &   5.2-7 $|$    3.5-7 $|$    7.7-7 $|$    {\red{2.6-4}} $|$    9.9-7 $|$    9.9-7 	 &00 $|$ 05 $|$ 11 $|$ 40 $|$ 46 $|$ 11\\[2pt] 
		392$;$3432& $10^{-4}$ & 128 	 &   6.1-7 $|$    7.5-7 $|$    {\red{ 3.3-5}} $|$    {\red{1.1-3}} $|$    {\red{2.3-4}} $|$    9.9-7 	 &00 $|$ 17 $|$ 56 $|$ 41 $|$ 50 $|$ 10\\[2pt] 
		\hline 
		space\_ga9 
		& $10^{-3}$ & 14 	 &   4.3-7 $|$    2.3-7 $|$    {\red{ 2.2-5}} $|$    {\red{ 9.3-5}} $|$    9.9-7 $|$    9.9-7 	 &01 $|$ 34 $|$ 03 $|$ 5:59 $|$ 4:18 $|$ 3:21\\[2pt] 
		3107$;$5005& $10^{-4}$ & 38 	 &   1.6-7 $|$    5.4-7 $|$    {\red{ 2.4-5}} $|$    {\red{2.3-3}} $|$    {\red{3.6-3}} $|$    9.9-7 	 &01 $|$ 49 $|$ 1:07 $|$ 6:00 $|$ 6:48 $|$ 3:02\\[2pt] 
		\hline 
		
	\end{longtable}
\end{scriptsize}

\subsection{Numerical results for Sparco collection}

In this subsection, the test instances $(\cA, b)$ are taken from $8$ real valued sparse reconstruction problems in the Sparco collection \cite{Berg2009testing}.
{For  testing} purpose,
we introduce a 60dB noise to $b$ (as in \cite{Fountoulakis14matrix})
{by}
using the {\sc Matlab} command:
{\tt b = awgn(b,60,'measured').}
For these test instances, the matrix representations of the linear maps $\cA$ are not available. Hence, ADMM will not be tested in this subsection,
{as it will be extremely expensive, if not impossible, to compute and factorize
$\cI+\sigma\cA\cA^*$.}

In Table \ref{table:sparco_details}, we report the detailed computational results obtained by {\sc Ssnal}, mfIPM, FPC\_AS, APG and LADMM in solving two large scale instances {\bf srcsep1} and {\bf srcsep2} in the Sparco collection. Here, we test five choices of $\lambda_c$, i.e., $\lambda_c = 10^{-0.8}, 10^{-1}, 10^{-1.2}, 10^{-1.5}, 10^{-2}$. As one can observe,
{as $\lambda_c$ decreases, the number of nonzeros (nnz) in the computed solution
increases.}
In the table, we list some statistics of the test instances, including the problem dimensions ($m,n$) and the largest eigenvalue of $\cA\cA^*$
($\lambda_{\max}(\cA\cA^*)$). For all the tested algorithms, we present the iteration %numbers,
{counts,}
the relative KKT residuals as well as the computation times (in the format {of} hours:minutes:seconds). One can observe from Table \ref{table:sparco_details} that all the algorithms perform very well for these easy-to-solve test instances. {As such,  \SSNAL does not have a clear advantage
as shown in Table \ref{table:UCI}. 
Moreover, since the matrix representations of the linear maps $\cA$ and $\cA^*$ involved are not stored explicitly (i.e., these linear maps can only be regarded as black-box functions), the second order sparsity can hardly be fully exploited. 
Nevertheless, our algorithm \SSNAL is generally faster than mfIPM, APG, LADMM while comparable with the fastest algorithm FPC\_AS. }

\begin{scriptsize}
	\begin{longtable}{| l| r|c| r|}
		\caption{The performance of {\sc Ssnal}, mfIPM, FPC\_AS, APG and LADMM on srcsep1 and srcsep2 (accuracy $\epsilon= 10^{-6}$, noise $60\textup{dB}$). $m$ is the sample size and $n$ is the dimension of features. In the table, ``a'' = {\sc Ssnal}, ``b''= mfIPM, ``c''= FPC\_AS, ``d'' = APG and
			``e''= LADMM, respectively. ``nnz'' denotes the number of nonzeros in the solution obtained by {\sc Ssnal}. The computation time is in the format of ``hours:minutes:seconds''.}\label{table:sparco_details}
		\\
		\hline
		\mc{1}{|c|}{} &\mc{1}{c|}{} &\mc{1}{c|}{} &\mc{1}{c|}{}\\[-5pt]
		\mc{1}{|c|}{} &\mc{1}{c|}{iteration} &\mc{1}{c|}{$\eta$}  &\mc{1}{c|}{time} \\[2pt] \hline
		\mc{1}{|c|}{$\lambda_c; {\rm nnz}$} &\mc{1}{c|}{a $|$ b $|$ c $|$ d $|$ e} &\mc{1}{c|}{a $|$ b $|$ c $|$ d $|$ e} &\mc{1}{c|}{a $|$ b $|$ c $|$ d $|$ e} \\ \hline
		\endhead
		\mc{4}{|c|}{} \\[-5pt]
				\mc{4}{|c|}{srcsep1,   $m=29166$, $n = 57344$, ${\lambda_{\max}(\cA\cA^*)} = 3.56$} \\[2pt] \hline
		$0.16$ $;$ 380& 14 $|$  28 $|$  42 $|$ 401 $|$  89	 &   6.4-7$|$   9.0-7 $|$    2.7-8 $|$    5.3-7 $|$    9.5-7 	 &07 $|$ 09 $|$ 05 $|$ 15 $|$ 07\\[2pt]
		\hline
		$0.10$ $;$ 726& 16 $|$  38 $|$  42 $|$ 574 $|$ 161	 &   4.7-7$|$   5.9-7 $|$    2.1-8 $|$    2.8-7 $|$    9.6-7 	 &11 $|$ 15 $|$ 05 $|$ 22 $|$ 13\\[2pt]
		\hline
		$0.06$ $;$ 1402& 19 $|$  41 $|$  63 $|$ 801 $|$ 393	 &   1.5-7$|$   1.5-7 $|$    5.4-8 $|$    6.3-7 $|$    9.9-7 	 &18 $|$ 18 $|$ 10 $|$ 30 $|$ 32\\[2pt]
		\hline
		$0.03$ $;$ 2899& 19 $|$  56 $|$ 110 $|$ 901 $|$ 337	 &   2.7-7$|$   9.4-7 $|$    7.3-8 $|$    9.3-7 $|$    9.9-7 	 &28 $|$ 53 $|$ 16 $|$ 33 $|$ 28\\[2pt]
		\hline
		$0.01$ $;$ 6538& 17 $|$  88 $|$ 223 $|$ 1401 $|$ 542	 &   7.1-7$|$   5.7-7 $|$    1.1-7 $|$    9.9-7 $|$    9.9-7 	 &1:21 $|$ 2:15 $|$ 34 $|$ 53 $|$ 45\\[2pt]
		\hline
				\mc{4}{|c|}{} \\[-5pt]
				\mc{4}{|c|}{srcsep2,   $m=29166$, $n = 86016$, ${\lambda_{\max}(\cA\cA^*)} = 4.95$} \\[2pt]
		\hline
		$0.16$ $;$ 423& 15 $|$  29 $|$  42 $|$ 501 $|$ 127	 &   3.2-7$|$   2.1-7 $|$    1.9-8 $|$    4.9-7 $|$    9.4-7 	 &14 $|$ 13 $|$ 07 $|$ 25 $|$ 15\\[2pt]
		\hline
		$0.10$ $;$ 805& 16 $|$  37 $|$  84 $|$ 601 $|$ 212	 &   8.8-7$|$   9.2-7 $|$    2.6-8 $|$    9.6-7 $|$    9.7-7 	 &21 $|$ 19 $|$ 16 $|$ 30 $|$ 26\\[2pt]
		\hline
		$0.06$ $;$ 1549& 19 $|$  40 $|$  84 $|$ 901 $|$ 419	 &   1.4-7$|$   3.1-7 $|$    5.4-8 $|$    4.7-7 $|$    9.9-7 	 &32 $|$ 28 $|$ 18 $|$ 44 $|$ 50\\[2pt]
		\hline
		$0.03$ $;$ 3254& 20 $|$  69 $|$ 128 $|$ 901 $|$ 488	 &   1.3-7$|$   6.6-7 $|$    8.9-7 $|$    9.4-7 $|$    9.9-7 	 &1:06 $|$ 1:33 $|$ 26 $|$ 44 $|$ 59\\[2pt]
		\hline
		$0.01$ $;$ 7400& 21 $|$  94 $|$ 259 $|$ 2201 $|$ 837	 &   8.8-7$|$   4.0-7 $|$    9.9-7 $|$    8.8-7 $|$    9.3-7 	 &1:42 $|$ 5:33 $|$ 59 $|$ 2:05 $|$ 1:43\\[2pt]
		\hline

	\end{longtable}
\end{scriptsize}

\bigskip

In Table \ref{table:sparco}, we report the numerical results obtained by {\sc Ssnal}, mfIPM, FPC\_AS, APG and LADMM in solving various instances of the Lasso problem (\ref{eq-l1}). For simplicity, we only test two cases with $\lambda_{c} = 10^{-3}$ and $10^{-4}$.
We can observe that FPC\_AS performs very well when it succeeds in obtaining a solution with the desired accuracy. However it is not robust in that it fails to solve 4 out  of 8 and 5 out of 8 problems when $\lambda_c = 10^{-3}$ and $10^{-4}$, respectively. For a few cases, FPC\_AS can only achieve a poor accuracy ($10^{-1}$).
The same 
{non-robustness}
also appears in the performance of APG.  This 
{non-robustness}
 is in fact closely related to the value of $\lambda_{\max}(\cA\cA^*)$. For example, both FPC\_AS and APG fail to solve a 
{rather}
small problem {\bf blknheavi} ($m = n =1024$) 
{whose}
 corresponding $\lambda_{\max}(\cA\cA^*) = 709$.
On the other hand,
LADMM, \SSNAL and mfIPM can solve all the test instances successfully. Nevertheless, in some cases, LADMM requires much more time than {\sc Ssnal}. One can also observe that for large scale problems, \SSNAL outperforms mfIPM by a large margin (sometimes up to a factor of 100). This also demonstrates the power of \SSNCG based augmented Lagrangian methods over interior-point methods in solving large scale problems.
Due to the high sensitivity of the first order algorithms
{to $\lambda_{\max}(\cA\cA^*)$}, one can safely conclude that the first order algorithms can only be used to solve relatively easy problems. Moreover, in order to obtain efficient and robust algorithms for Lasso problems or more general convex composite optimization problems, it is necessary to carefully %explore
{exploit}
the second order information in the algorithmic design.

\begin{scriptsize}
	\begin{longtable}{|  l  | c | c | c| r|}
		\caption{The performance of {\sc Ssnal}, mfIPM, FPC\_AS, APG and LADMM on 8 selected sparco problems (accuracy $\epsilon= 10^{-6}$, noise $60\textup{dB}$). $m$ is the sample size and $n$ is the dimension of features. In the table, ``a'' = {\sc Ssnal}, ``b'' = mfIPM, ``c'' = FPC\_AS, ``d'' = APG and ``e'' = LADMM, respectively. ``nnz'' denotes the number of nonzeros in the solution obtained by {\sc Ssnal}. The computation time is in the format of ``hours:minutes:seconds''.}\label{table:sparco}
		\\
		\hline
		\mc{1}{|c|}{} &\mc{1}{c|}{} &\mc{1}{c|}{} &\mc{1}{c|}{}&\mc{1}{c|}{}\\[-5pt]
		\mc{1}{|c|}{} & \mc{1}{c|}{$\lambda_c$} &\mc{1}{c|}{nnz} &\mc{1}{c|}{$\eta$} &\mc{1}{c|}{time} \\[2pt] \hline
		\mc{1}{|c|}{probname}  &\mc{1}{c|}{} &\mc{1}{c|}{}
		&\mc{1}{c|}{a $|$ b $|$ c $|$ d $|$ e  } &\mc{1}{c|}{a $|$ b $|$ c $|$ d $|$ e } \\[2pt]
		\mc{1}{|c|}{$m;n$} &\mc{1}{c|}{ } &\mc{1}{c|}{}
		&\mc{1}{c|}{ } &\mc{1}{c|}{ } \\ \hline
		\endhead
blknheavi
& $10^{-3}$ & 12 	 &   5.7-7 $|$    9.2-7 $|$    {\red{1.3-1}} $|$    {\red{ 2.0-6}} $|$    9.9-7 	 &01 $|$ 01 $|$ 55 $|$ 08 $|$ 06\\[2pt]
1024$;$1024& $10^{-4}$ & 12 	 &   9.2-8 $|$    8.7-7 $|$    {\red{4.6-3}} $|$    {\red{ 8.5-5}} $|$    9.9-7 	 &01 $|$ 01 $|$ 49 $|$ 07 $|$ 07\\[2pt]
\hline
srcsep1
& $10^{-3}$ & 14066 	 &   1.6-7 $|$    7.3-7 $|$    9.7-7 $|$    8.7-7 $|$    9.7-7 	 &5:41 $|$ 42:34 $|$ 13:25 $|$ 1:56 $|$ 4:16\\[2pt]
29166$;$57344& $10^{-4}$ & 19306 	 &   9.8-7 $|$    9.5-7 $|$    9.9-7 $|$    9.9-7 $|$    9.5-7 	 &9:28 $|$   3:31:08 $|$ 32:28 $|$ 2:50 $|$ 13:06\\[2pt]
\hline
srcsep2
& $10^{-3}$ & 16502 	 &   3.9-7 $|$    6.8-7 $|$    9.9-7 $|$    9.7-7 $|$    9.8-7 	 &9:51 $|$   1:01:10 $|$ 16:27 $|$ 2:57 $|$ 8:49\\[2pt]
29166$;$86016& $10^{-4}$ & 22315 	 &   7.9-7 $|$    9.5-7 $|$    {\red{1.0-3}} $|$    9.6-7 $|$    9.5-7 	 &19:14 $|$   6:40:21 $|$   2:01:06 $|$ 4:56 $|$ 16:01\\[2pt]
\hline
srcsep3
& $10^{-3}$ & 27314 	 &   6.1-7 $|$    9.6-7 $|$    9.9-7 $|$    9.6-7 $|$    9.9-7 	 &33 $|$ 6:24 $|$ 8:51 $|$ 47 $|$ 49\\[2pt]
196608$;$196608& $10^{-4}$ & 83785 	 &   9.7-7 $|$    9.9-7 $|$    9.9-7 $|$    9.9-7 $|$    9.9-7 	 &2:03 $|$   1:42:59 $|$ 3:40 $|$ 1:15 $|$ 3:06\\[2pt]
\hline
soccer1
& $10^{-3}$ &  4 	 &   1.8-7 $|$    6.3-7 $|$    {\red{5.2-1}} $|$    8.4-7 $|$    9.9-7 	 &01 $|$ 03 $|$ 13:51 $|$ 2:35 $|$ 02\\[2pt]
3200$;$4096& $10^{-4}$ &  8 	 &   8.7-7 $|$    4.3-7 $|$    {\red{5.2-1}} $|$    {\red{ 3.3-6}} $|$    9.6-7 	 &01 $|$ 02 $|$ 13:23 $|$ 3:07 $|$ 02\\[2pt]
\hline
soccer2
& $10^{-3}$ &  4 	 &   3.4-7 $|$    6.3-7 $|$    {\red{5.0-1}} $|$    8.2-7 $|$    9.9-7 	 &00 $|$ 03 $|$ 13:46 $|$ 1:40 $|$ 02\\[2pt]
3200$;$4096& $10^{-4}$ &  8 	 &   2.1-7 $|$    1.4-7 $|$    {\red{6.8-1}} $|$    {\red{ 1.8-6}} $|$    9.1-7 	 &01 $|$ 03 $|$ 13:27 $|$ 3:07 $|$ 02\\[2pt]
\hline
blurrycam
& $10^{-3}$ & 1694 	 &   1.9-7 $|$    6.5-7 $|$    3.6-8 $|$    4.1-7 $|$    9.4-7 	 &03 $|$ 09 $|$ 03 $|$ 02 $|$ 07\\[2pt]
65536$;$65536& $10^{-4}$ & 5630 	 &   1.0-7 $|$    9.7-7 $|$    1.3-7 $|$    9.7-7 $|$    9.9-7 	 &05 $|$ 1:35 $|$ 08 $|$ 03 $|$ 29\\[2pt]
\hline
blurspike
& $10^{-3}$ & 1954 	 &   3.1-7 $|$    9.5-7 $|$    {\red{7.4-4}} $|$    9.9-7 $|$    9.9-7 	 &03 $|$ 05 $|$ 6:38 $|$ 03 $|$ 27\\[2pt]
16384$;$16384& $10^{-4}$ & 11698 	 &   3.5-7 $|$    7.4-7 $|$    {\red{ 8.3-5}} $|$    9.8-7 $|$    9.9-7 	 &10 $|$ 08 $|$ 6:43 $|$ 05 $|$ 35\\[2pt]
\hline

	\end{longtable}
\end{scriptsize}

\section{Conclusion}
In this paper, we have proposed an inexact augmented Lagrangian method of an asymptotic superlinear convergence rate for solving the large scale convex composite optimization problems of the form ({\bf P}).
It is particularly well suited for solving $\ell_1$-regularized {least squares (LS)} problems.
With the {intelligent} incorporation of the semismooth Newton method, our algorithm \SSNAL is able to fully exploit the second order sparsity of the problems.
Numerical results have convincingly demonstrated the superior efficiency and robustness
of our algorithm in
solving large scale $\ell_1$-regularized LS problems.
Based on extensive numerical evidence, we firmly believe that our
algorithmic framework can be adapted to design robust and efficient solvers for various large scale
{convex composite} optimization problems.

\section*{ Acknowledgments}
{The authors would like to thank the anonymous referees for carefully reading our work and for their helpful suggestions.}
The authors would also like to thank Dr. Ying Cui at National University of Singapore and Dr. Chao Ding at Chinese Academy of Sciences for numerous discussions on the error bound conditions and the metric subregularity.

\small
{

}
\end{document}